\documentclass[11]{siamltex}

\usepackage[pagewise]{lineno} 

\newcommand{\nwc}{\newcommand}
\usepackage[pdftex]{graphicx}
\usepackage{psfrag}
\usepackage{wrapfig}
\usepackage{epsf}
\usepackage{amsmath,amssymb}
\usepackage[font={footnotesize}]{caption} 
\usepackage{url}
\usepackage[mathscr]{euscript}
\usepackage{xcolor}
\usepackage{subcaption}
\usepackage{tikz}
\usetikzlibrary{patterns}
\usepackage{pgfplots}
\tikzset{every picture/.style={line width=0.75pt}} 

\renewcommand{\theequation}{\arabic{equation}}

\newtheorem{thm}{Theorem}[section]

\newtheorem{rema}[thm]{Remark}
\newtheorem{result}[thm]{Main Result}

\newcommand{\barint}{\hbox{$\int$\kern-0.75\intwidth
\vrule width 0.5\intwidth height 2.4pt depth -2pt\kern0.25\intwidth}}
\newlength\intwidth
\setbox0=\hbox{$\int$}
\intwidth=\wd0

\newcommand\avint{\hbox{\hbox{$\displaystyle \int$}\hbox{\kern-.9em{$-$}}}}

\newcommand\smavint{\hbox{\hbox{$\int$}\hbox{\kern-.75em{$-$}}}}

\nwc{\st}{^{\mbox{\it st}}}

\nwc{\qref}[1]{(\ref{#1})}
\nwc{\veloc}{v}
\nwc{\rhoc}{\beta}
\nwc{\hl}{\hat{L}}

\def\Xint#1{\mathchoice
{\XXint\displaystyle\textstyle{#1}}%
{\XXint\textstyle\scriptstyle{#1}}%
{\XXint\scriptstyle\scriptscriptstyle{#1}}%
{\XXint\scriptscriptstyle\scriptscriptstyle{#1}}%
\!\int}
\def\XXint#1#2#3{{\setbox0=\hbox{$#1{#2#3}{\int}$}
\vcenter{\hbox{$#2#3$}}\kern-.51\wd0}}

\def\dashint{\Xint-}

\nwc{\intRp}{\int_0^\infty}
\nwc{\aint}{\dashint}
\nwc{\aaint}{\dashint}

\newcommand{\cE}{{\cal E}}

\newcommand{\R}{\mathbb R}

\newcommand{\be}{\begin{eqnarray}}
\newcommand{\ee}{\end{eqnarray}}

\newcommand{\ben}{\begin{eqnarray*}}
\newcommand{\een}{\end{eqnarray*}}

\newcommand{\x}[2]{\|\, #1 \,\|_{#2}}

\newcommand\restr[2]{{
		\left.\kern-\nulldelimiterspace 
		#1 
		\littletaller 
		\right|_{#2} 
}}

\newcommand{\littletaller}{\mathchoice{\vphantom{\big|}}{}{}{}}

\title{}
\author{}

\numberwithin{equation}{section}
\usepackage{indentfirst}
\usepackage{stmaryrd}

\begin{document}
	
\title{A Diffuse Domain Approximation with \\ Transmission-Type Boundary Conditions I: \\ Asymptotic Analysis and Numerics}
\author{
	Toai Luong
	\thanks{Corresponding author, Department of Mathematics, The University of Tennessee, Knoxville, TN, USA. 
		Email: tluong4@utk.edu.
	} \and 
	Tadele Mengesha
	\thanks{Department of Mathematics, The University of Tennessee, Knoxville, TN, USA. 
		Email: mengesha@utk.edu.
	} \and 
	Steven M. Wise
	\thanks{Department of Mathematics, The University of Tennessee, Knoxville, TN, USA. 
	Email: swise1@utk.edu
	} \and 
	Ming Hei Wong
	\thanks{Department of Mathematics, The University of Tennessee, Knoxville, TN, USA. 
		Email: mwong4@vols.utk.edu.
	}
}

\maketitle

    \begin{abstract}
Diffuse domain methods (DDMs) have garnered significant attention for approximating solutions to partial differential equations on complex geometries. These methods implicitly represent the geometry by replacing the sharp boundary interface with a diffuse layer of thickness $\varepsilon$, which scales with the minimum grid size. This approach reformulates the original equations on an extended regular domain, incorporating boundary conditions through singular source terms. In this work, we conduct a matched asymptotic analysis of a DDM for a two-sided problem with transmission-type Robin boundary conditions. Our results show that, in the one dimensional space, the solution of the diffuse domain approximation asymptotically converges to the solution of the original problem, with exactly first-order accuracy in $\varepsilon$. Furthermore, we provide numerical simulations that validate and illustrate the analytical result.  
    \end{abstract}

\begin{keywords}
	partial differential equations, phase-field approximation, diffuse domain method, diffuse interface approximation, asymptotic analysis, numerical simulation, reaction-diffusion equation, transmission boundary conditions.
\end{keywords}

    \section{Introduction}

Partial differential equations (PDEs) are foundational tools for modeling diverse phenomena across physical, biological, and engineering sciences, including fluid flow, material behavior, tissue dynamics, and phase transitions. In many practical scenarios, these problems arise within domains that are complex, irregular, or time-dependent, such as evolving interfaces in phase transitions or intricate geometries in biological systems. Traditional PDE solution methods often require domains with simple, specific geometric boundaries, posing challenges in mesh generation and driving up computational costs. To address these limitations, diffuse domain methods (DDMs) have emerged as versatile approaches for solving PDEs on irregular or dynamically evolving domains.

The fundamental principles of DDMs involve (i) embedding the original complex domain into a larger, simpler computational domain, like a square or a cube, (ii) creating a simple, structured mesh for the larger domain that resolves the shape of the complex domain without fitting it precisely, and (iii) solving an approximate PDE problem on the larger computational domain. This approach eliminates the need for intricate, boundary-conforming meshes that conventional methods typically require.  This is especially important in the time-dependent setting where the shape of the complex domain is constantly changing, requiring expensive re-meshing. A smooth phase-field function is employed to approximate the characteristic function of the complex domain, while a parameter $\varepsilon$, typically related to the grid size, defines the width of the diffuse interfacial region, influencing the accuracy of the approximation. The original PDE is reformulated with additional source terms to enforce boundary conditions. For small values of $\varepsilon$, DDMs are especially efficient when paired with adaptive mesh refinement, which allows for fine grid cells in the narrow transition layer and coarser cells in the extended, non-physical regions of the domain. DDMs offer the advantage of flexible application to a broad range of equations and can be solved using standard discretization methods, both uniform and adaptive, along with fast iterative solvers based, for example, on geometric multigrid methods.

The concept of the DDM was first introduced by Kockelkoren and Levine \cite{Kockelkoren-DDM2003} for studying diffusion within a cell with zero Neumann boundary conditions at the cell boundary. 
A related approach, known as the fictitious domain method, was earlier employed by Glowinski et al. \cite{GlowinskiFDD-1992, GlowinskiFDD-1994} to compute numerical solutions to Dirichlet problems for a class of elliptic operators. 
Since then, DDMs have been subsequently applied to model electrical waves in the heart \cite{Fenton-DDM2005} and membrane-bound Turing patterns \cite{Levine-DDM2005}.  More recent developments have expanded DDMs to solve PDEs on both stationary \cite{Voigt-DDM2006} and evolving surfaces (e.g., \cite{Demlow-DDM2007, Dziuk-DDM2008, Dziuk-DDM2010, Dziuk-DDM2012, Elliott-DDM2009, Elliott-DDM2011}). The analysis of DDM approach to solving elliptic PDEs in domains with complex boundaries subject to  Dirichlet, Neumann, and Robin boundary conditions is provided in \cite{MR3648102,MR3540309,MR3032706,Li-DDM2006,MR3541946}.
DDMs have gained wide use in applications such as phase-field modeling, where they support simulations of complex phenomena in fields like biology (e.g., \cite{Camley-DDM2013, Kockelkoren-DDM2003, Fenton-DDM2005, Lowengrub-DDM2014, Ratz-DDM2015, Voigt-DDM2011-bio}), 
fluid dynamics (e.g., \cite{Voigt-DDM2010, Voigt-DDM2011-fluid, Voigt-DDM2012, Voigt-DDM2014,MR4642032,Voigt-DDM2009}), and materials science (e.g., \cite{Thornton-DDM2016, Ratz-DDM2015, Ratz-DDM2016, Thornton-DDM2018}).

    \subsection{The One-sided Diffuse Domain Problem}

Let $\Omega_1$ be a bounded open subset of $\R^n$. We consider the reaction-diffusion equation in $\Omega_1$: 
    \begin{align}
    \label{intro-eqn}
-\Delta u + u & = f,  \quad \text{in } \Omega_1,  
    \\
    \label{intro-bcs}
- \nabla u \cdot \boldsymbol{n}_1 & = \kappa u + g, \quad \text{on } \partial\Omega_1,
    \end{align}
where $\boldsymbol{n}_1$ denotes the outward-pointing  unit normal  vector on $\partial\Omega_1$. 
Here, $\kappa \geq 0$ is a given constant. 
Observe that Neumann boundary conditions hold when $\kappa = 0$, and Robin boundary conditions hold when $\kappa > 0$.

To approximate this problem using a diffuse domain approach, we define an extended domain $\Omega$, a larger cuboidal region containing $\overline{\Omega_1}$ (see Figure~\ref{domain}). 
In this extended domain $\Omega$, the diffuse domain approximation equation is 
    \begin{align}\label{intro-diff-eqn}
-\nabla \cdot (\phi_\varepsilon \nabla u_\varepsilon) + \phi_\varepsilon u_\varepsilon + \text{BC} = \phi_\varepsilon f, 
    \end{align}
where $\phi_\varepsilon(x)$ approximates the characteristic function $\chi_{\Omega_1}(x)$ of $\Omega_1$, given by 
\begin{align*}
	\chi_{\Omega_1}(x) = \begin{cases}
		1,  &\text{if } x \in \Omega_1, \\
		0,  &\text{if } x \notin \Omega_1.
	\end{cases}
\end{align*}

\begin{figure}[htb!]
    \centering
    \begin{tikzpicture}[x=0.75pt,y=0.75pt,yscale=-1,xscale=1]

\draw  [line width=2.25]  (272,46) -- (423.67,46) -- (423.67,198.17) -- (272,198.17) -- cycle ;
\draw  [fill={rgb, 255:red, 184; green, 178; blue, 178 }  ,fill opacity=1 ] (380,107.33) .. controls (417,84.33) and (403,159.33) .. (348,163.33) .. controls (293,167.33) and (284,82.33) .. (314,96.33) .. controls (344,110.33) and (343,130.33) .. (380,107.33) -- cycle ;
\draw [line width=0.75]    (325.67,103) -- (338.16,81.59) ;
\draw [shift={(339.67,79)}, rotate = 120.26] [fill={rgb, 255:red, 0; green, 0; blue, 0 }  ][line width=0.08]  [draw opacity=0] (6.25,-3) -- (0,0) -- (6.25,3) -- cycle    ;
\draw [line width=0.75]    (424.17,120) -- (448.33,120.22) ;
\draw [shift={(451.33,120.25)}, rotate = 180.53] [fill={rgb, 255:red, 0; green, 0; blue, 0 }  ][line width=0.08]  [draw opacity=0] (6.25,-3) -- (0,0) -- (6.25,3) -- cycle    ;

\draw (337,130) node [anchor=north west][inner sep=0.75pt]    {$\Omega _{1}$};
\draw (290,165) node [anchor=north west][inner sep=0.75pt]    {$\Omega _{2}$};
\draw (342.33,68) node [anchor=north west][inner sep=0.75pt]    {$\boldsymbol{n}_{1}$};
\draw (454.83,115) node [anchor=north west][inner sep=0.75pt]    {$\boldsymbol{n}_{2}$};

\end{tikzpicture}
    \caption{A domain $\Omega_1$ is covered by a larger cuboidal domain $\Omega$, with $\Omega_2 := \Omega \setminus \overline{\Omega_1}$.}
	\label{domain}
\end{figure}

A common approximation for $\chi_{\Omega_1}$ is the phase-field function
\begin{align*}
	\phi_\varepsilon(x) := \frac{1}{2}\left[ 1 + \tanh \left( \frac{r(x)}{\varepsilon}\right) \right] \approx \chi_{\Omega_1}(x),
\end{align*} 
where $\varepsilon > 0$ is small, defining the interface thickness. 
Here, $r(x)$ is the signed distance function from $x \in \R^n$ to $\partial\Omega_1$, 
which is assumed to be positive within $\Omega_1$ and negative outside $\overline{\Omega_1}$. 
There are different choices for the boundary term BC in \qref{intro-diff-eqn} \cite{Li-DDM2009, Lowengrub-DDM2015}. 
For instance, we may choose either 
    \begin{align}
    \label{BC1}
\text{BC} = \text{BC1} = (\kappa u_\varepsilon + g)|\nabla\phi_\varepsilon|,
    \end{align}
or 
    \begin{align}
    \label{BC2}
\text{BC} = \text{BC2} = (\kappa u_\varepsilon + g)\varepsilon|\nabla\phi_\varepsilon|^2,
    \end{align}
where $|\nabla\phi_\varepsilon(x)|$ and $\varepsilon|\nabla\phi_\varepsilon(x)|^2$ approximate the surface delta function $\delta_{\partial\Omega_1}$ of $\partial\Omega_1$.

Using asymptotic analysis for the Neumann boundary condition case ($\kappa = 0$), Lervåg and Lowengrub~\cite{Lowengrub-DDM2015} argued that Equation~\qref{intro-diff-eqn} with either BC1 or BC2 is second-order accurate in $\varepsilon$. Specifically, in the expansion of the solution,
    \begin{align*}
u_\varepsilon(x) = u_0(x) +\varepsilon u_1(x) + \varepsilon^2 u_2(x) + \cdots ,
    \end{align*}
they argued that $u_1\equiv 0$, and $u_0$ solves \eqref{intro-eqn}--\eqref{intro-bcs}, so that
    \begin{align*}
u_\varepsilon(x) - u_0(x) = O(\varepsilon^2).
    \end{align*}
However, their asymptotic analysis was limited to the interior domain $\Omega_1$. Our objective is to extend their matched asymptotic analysis to encompass both the interior domain $\Omega_1$ and the exterior domain $\Omega_2 := \Omega \setminus \overline{\Omega_1}$. 

In the exterior domain $\Omega_2$, the function $\phi_\varepsilon$ decays exponentially, which introduces significant challenges for both theoretical and numerical analyses. Specifically, the stiffness matrices in the numerical methods would have diagonal entries that are exponentially small, making numerical approximation impractical.  To avoid ill-conditioning in practical numerical simulations, Lerv{\aa}g and Lowengrub~\cite{Lowengrub-DDM2015} introduced a modified phase-field approximation for the highest-order term, replacing $-\nabla \cdot(\phi_\varepsilon\nabla u_\varepsilon)$ with
    \begin{align*}
-\nabla \cdot(D_\varepsilon\nabla u_\varepsilon) , \quad D_\varepsilon(x) := \alpha + (1 - \alpha)\phi_\varepsilon(x),
    \end{align*}
where $\alpha$ is a very small fixed positive constant (specifically, $\alpha = 10^{-6}$), converting the problem into what is effectively a two-sided formulation. Note that, with this change, some further information about what problem $u_\varepsilon$ should satisfy in $\Omega_2$ is needed. For one thing, boundary conditions are required on the outer boundary $\partial\Omega$. We emphasize that, while they used this two-sided formulation in their numerical simulations, their asymptotic analysis covers only the one-sided problem, thus leaving a gap between the analysis and the practical implementation.  Our ultimate goal is to perform asymptotic analyses on the two-sided approximation with respect to smallness of both $\varepsilon$ and the stabilizing parameter $\alpha$, leading to a more complete understanding of the approximation of the one-sided problem. In this paper, however, we only take a preliminary step by performing asymptotic analyses and numerical simulations for the two-sided problem, for the fixed finite $\alpha$ case, though in a slightly more generalized form, described in Section~\ref{sec:two-sided}. See also the discussion section, Section~\ref{sec:discuss}, at the end of the paper.

\subsection{The Two-sided Diffuse Domain Problem}\label{sec:two-sided}
Let $\Omega_1$ be a bounded open subset of $\R^n$ with a sufficiently smooth boundary $\partial\Omega_1$. 
Let $\Omega$ be an open cuboidal domain such that $\Omega \supset \overline{\Omega_1}$ and $\partial\Omega \cap \partial\Omega_1 = \varnothing$. 
Define $\Omega_2 := \Omega \setminus \overline{\Omega_1}$. 
We consider the following two-sided boundary value problem in $\Omega$: 
Find a pair of functions, $u_1:\Omega_1 \to \R$ and $u_2:\Omega_2 \to \R$, that satisfy
    \begin{align}
    \label{bvp1} 
-\Delta u_1 + \gamma u_1 &= q, \quad \text{in } \Omega_1,  
    \\
	\label{bvp2} 
-\alpha\Delta u_2 + \beta u_2 &= h, \quad \text{in } \Omega_2, 
    \\
	\label{bvp3} 
 u_1 &= u_2, \quad \text{on } \partial\Omega_1,  
    \\
	\label{bvp4} 
-\boldsymbol{n}_1\cdot \nabla (u_1 - \alpha u_2)  & = \kappa u_1 + g, \quad \text{on } \partial\Omega_1, 
    \\
	\label{bvp5} 
 \alpha \, \boldsymbol{n}_2 \cdot \nabla u_2  & = 0, \quad \text{on } \partial\Omega.
    \end{align}
The boundary conditions across the interface $\partial\Omega_1$ are called transmission-type boundary conditions. These guarantee that the function values are continuous across the interface, and the fluxes have a jump discontinuity owing to some physical phenomenon. For example, in electrostatics problems, excess charge on the interface can result in such jump conditions for the flux.

Here, we assume the following:
    \begin{enumerate}
	\item[(1)] 
$h, q \in L^2(\Omega)$ and $g \in H^1(\Omega)$ are given functions;
	
	\item[(2)] 
$\alpha, \beta, \gamma$ are given positive constants, and $\kappa$ is a given nonnegative constant;
	
	\item[(3)]  
$\boldsymbol{n}_1$ denotes the outward-pointing  unit normal  vector on $\partial\Omega_1$, and $\boldsymbol{n}_2$ denotes the outward-pointing  unit normal vector on $\partial\Omega$ (see Figure~\ref{domain}).

    \end{enumerate}

A solution to the two-sided problem \qref{bvp1}--\qref{bvp5} is understood in the weak sense. 
If $(u_1, u_2) \in H^1(\Omega_1) \times H^1(\Omega_2)$ is a solution pair to the two-sided problem \qref{bvp1}--\qref{bvp5}, then $u_0: \Omega \to \R$, defined as
\begin{align}\label{defn-u0}
	u_0(x) := \begin{cases}
		u_1(x),  &\text{if } x \in \Omega_1, \\
		u_2(x),  &\text{if } x \in \Omega_2,
	\end{cases}
\end{align}
belongs to  $H^1(\Omega)$ and solves the following weak formulation:
    \begin{align} \label{weak-form}
\int_\Omega (D_0 \nabla u_0 \cdot \nabla w + c_0 u_0 w - f_0 w) \; dx + \int_{\partial\Omega_1}(\kappa u_0 + g)w\; dS = 0,
    \end{align}
for any $w \in H^1(\Omega)$,
where
\begin{align*}
	D_0(x) &:= \chi_{\Omega_1}(x) + \alpha\chi_{\Omega_2}(x), \\
	c_0(x) &:= \gamma\chi_{\Omega_1}(x) + \beta\chi_{\Omega_2}(x), \\
	f_0(x) &:= q(x)\chi_{\Omega_1}(x) + h(x)\chi_{\Omega_2}(x).	
\end{align*} 
Equivalently, $u_0$ minimizes the energy functional, $\cE_0$, given by
\begin{align*}
	\cE_0[u] = \int_{\Omega} \left[ \frac{1}{2}(D_0 |\nabla u|^2 + c_0 u^2) - f_0 u \right]\; dx + \int_{\partial\Omega_1} \left( \frac{1}{2}\kappa u^2 + gu \right) \; dS, \quad u \in H^1(\Omega), 
\end{align*}
Conversely, if $u_0$ minimizes $\cE_0$ over $H^1(\Omega)$, then $u_1: = \restr{u_0}{\Omega_1}$ and $u_2: = \restr{u_0}{\Omega_2}$ solve the two-sided problem \qref{bvp1}--\qref{bvp5}. 

Since $\cE_0$ is coercive and strictly convex, it has a unique minimizer $u_0 \in H^1(\Omega)$, which implies that the two-sided problem \qref{bvp1}--\qref{bvp5} has a unique solution, which is identified by $u_0$ via \eqref{defn-u0}.

For each $\varepsilon \in (0,1)$, using the boundary term BC1 defined by \qref{BC1}, the diffuse domain approximation of the problem \qref{bvp1}--\qref{bvp5} is then given by: 
Find a function $u_\varepsilon : \Omega \to \R$ that satisfies
    \begin{align}
	\label{diff-bvp1} 
-\nabla \cdot (D_\varepsilon \nabla u_\varepsilon) + c_\varepsilon u_\varepsilon + (\kappa u_\varepsilon + g)|\nabla\phi_\varepsilon| &= f_\varepsilon, \quad \text{in } \Omega,  
    \\
	\label{diff-bvp2} 
D_\varepsilon \nabla u_\varepsilon \cdot \boldsymbol{n}_2 & = 0, \quad \text{on } \partial\Omega,
    \end{align}
where
    \begin{align*}
D_\varepsilon(x) &:= \alpha + (1 - \alpha)\phi_\varepsilon(x) \approx D_0(x), 
    \\
c_\varepsilon(x) &:= \beta + (\gamma - \beta)\phi_\varepsilon(x) \approx c_0(x), 
    \\
f_\varepsilon(x) &:= h(x) + [q(x) - h(x)]\phi_\varepsilon(x) \approx f_0(x).	
    \end{align*}
For each $\varepsilon$, a solution $u_\varepsilon$ of the problem \qref{diff-bvp1}--\qref{diff-bvp2} minimizes the associated energy functional, $\cE_\varepsilon$, defined by
    \begin{align*}
\cE_\varepsilon [u] = \int_\Omega \left[ \frac{1}{2} (D_\varepsilon |\nabla u|^2 + c_\varepsilon u^2) - f_\varepsilon u +  \left(\frac{1}{2} \kappa u^2 + gu \right) |\nabla\phi_\varepsilon| \right]\; dx, \quad u \in H^1(\Omega). 
    \end{align*}
Since $\cE_\varepsilon$ is coercive and strictly convex, it has a unique minimizer, which implies that the diffuse domain problem \qref{diff-bvp1}--\qref{diff-bvp2} has a unique solution $u_\varepsilon \in H^1(\Omega)$.

\subsection{A Singular $\alpha \to 0$ Limit of the Two-Sided Problem}

One may ask \emph{How is the one-sided problem related to the two-sided version?} It is not a simple matter of setting $\alpha = 0$. Indeed, the limit $\alpha\to 0$ in the two-sided problem is, in general, singular. However, there is a setting for which the singular limit makes sense and yields a meaningful solution. 

Suppose that $\beta =0$ and $h\equiv 0$. Then we claim that the following decoupled problem emerges in the $\alpha\to 0$ limit: Find a pair of functions, $u_1:\Omega_1 \to \R$ and $u_2:\Omega_2 \to \R$, that satisfy
    \begin{align}
    \label{bvp1-sing} 
-\Delta u_1 + \gamma u_1 &= q, \quad \text{in } \Omega_1,  
    \\
	\label{bvp2-sing} 
-\Delta u_2  &= 0, \quad \text{in } \Omega_2, 
    \\
	\label{bvp3-sing} 
 u_1 &= u_2, \quad \text{on } \partial\Omega_1,  
    \\
	\label{bvp4-sing} 
- \boldsymbol{n}_1\cdot  \nabla u_1  & = \kappa u_1 + g, \quad \text{on } \partial\Omega_1, 
    \\
	\label{bvp5-sing} 
\boldsymbol{n}_2 \cdot \nabla u_2  & = 0, \quad \text{on } \partial\Omega.
    \end{align}
 This is just the one-sided problem coupled to a Laplace-type problem on an annular domain. The interior problem involving $u_1$ can be solved first. Subsequently, $u_2$ can be obtained by solving a Laplace's equation, with Dirichlet boundary conditions on $\partial\Omega_1$, namely, $u_2 = u_1$, and homogeneous Neumann conditions on $\partial\Omega_2$. This realization will lead us to a new diffuse-domain approximation method for the one-sided problem, as we explain in Section~\ref{sec:discuss}. The rigorous justification of the claim above is reserved for a future paper. For now, we will focus on the two-sided problem.

\section{Main Result}\label{sec:main-result-1}
In this paper, we study the asymptotic convergence of the diffuse domain approximation problem \eqref{diff-bvp1}--\eqref{diff-bvp2}  in the one dimensional space. 
Let $u_0$ be the solution of the two-sided problem \qref{bvp1}--\qref{bvp5} defined by \qref{defn-u0}, 
and let $u_\varepsilon$ be the solution to the diffuse domain approximation problem \qref{diff-bvp1}--\qref{diff-bvp2}, for each $\varepsilon \in (0,1)$. 
Assuming that the functions $h(x)$ and $q(x)$ are analytic, 
and the function $g(x)$ remains constant in the normal direction to the boundary $\partial\Omega_1$ of $\Omega_1$ within a narrow strip around $\partial\Omega_1$, 
we perform a matched asymptotic analysis for the problem \qref{diff-bvp1}--\qref{diff-bvp2} in 1D, leading to the main result presented below. 
Additionally, we provide numerical simulations and discuss their outcomes in relation to our analytical result.

    \begin{result}[Asymptotic convergence of $u_\varepsilon$]
    \label{thm-main-1}
In one dimension, $u_\varepsilon$ converges asymptotically and uniformly to $u_0$, as $\varepsilon \to 0$. Moreover, the diffuse domain approximation is precisely first-order accurate in $\varepsilon$, as expressed in \eqref{eqn:basic-approx}. 
    \end{result}

    \begin{rema}
We conjecture that the main result can be generalized to higher-dimensional spaces, provided the interface is sufficiently smooth. However, the analysis will become more intricate, as it involves the interface curvature and the Laplace-Beltrami operator. 
    \end{rema}

\begin{rema}
    The asymptotic analysis in Section~\ref{sec:proof-thm-1} considers the case where $g(x)$ is constant in the normal direction to $\partial\Omega_1$ within a narrow strip around $\partial\Omega_1$. In one dimension, this simply means that $g(x)$ is constant in some neighborhood of $\partial\Omega_1$, as we shall see.    Nonetheless, we expect that the asymptotic analysis can be adjusted to accommodate the case where $g(x)$ varies in the normal direction within this strip. Additionally, our numerical experiments for this more general setting suggest that the convergence rate remains first-order.
\end{rema}

\section{Asymptotic Analysis for the Diffuse Domain Problem in 1D}\label{sec:proof-thm-1}
In this section, we demonstrate a matched asymptotic analysis for the diffuse domain approximation problem \qref{diff-bvp1}--\qref{diff-bvp2} in 1D, 
which establishes the main result in Section~\ref{sec:main-result-1}. 
Without loss of generality, we can simplify the structure of the domains. In particular, we consider the domains as follows:
\begin{align*}
    \Omega = (-1,1), \quad \Omega_1 = \Omega_R = (0,1), \quad \Omega_2 = \Omega_L = (-1,0).
\end{align*}
The phase-field function is then defined as
\begin{align*}
	\phi_\varepsilon (x) = \frac{1}{2}\left[ 1 + \tanh \left(\frac{x}{\varepsilon}\right)\right] \approx \chi_{\Omega_R}(x),
\end{align*}
for each $\varepsilon \in (0,1)$ (see Figure~\ref{domain-1D}).

\begin{figure}[htb!]
	\centering
	\begin{tikzpicture}
    \begin{axis}[axis equal image,
        axis lines=middle,
        xlabel={$x$},
        ylabel={$y$},
        xmin=-1.2, xmax=1.2,
        ymin=-0.2, ymax=1.2,
        xtick={-1,0,1},
        ytick={0,1},
        legend pos=south east,
        domain=-1:1,
        samples=200,
        smooth
    ]
        \addplot[thick,red] {0.5 * (1 + tanh(x / 0.1))};
        \node at (axis cs:-0.5,-0.05) [below] {$\Omega_L$};
        \node at (axis cs:0.5,-0.05) [below] {$\Omega_R$};
        \node at (axis cs:0.5,1) [above] {$\phi_\varepsilon(x)$};
    \end{axis}
\end{tikzpicture}
	\caption{The domains $\Omega = (-1,1)$, $\Omega_L = (-1,0)$, $\Omega_R = (0,1)$, and the graph of $\phi_\varepsilon (x)$ in 1D.}
	\label{domain-1D}
\end{figure}

We assume that the functions $h(x)$ and $q(x)$ are analytic.
For convenience, we further assume that $g(x) \equiv \lambda$ in a small interval $(-\zeta, \zeta)$ around $x = 0$, where $\lambda$ is a constant.
Let $u_R:\Omega_R \to \R$ and $u_L:\Omega_L \to \R$ be the pair of functions that satisfy the following two-sided problem in 1D:
    \begin{align}
	\label{1D-bvp1} 
-u_R'' + \gamma u_R &= q, \quad \text{in } \Omega_R, 
    \\
	\label{1D-bvp2} 
-\alpha u_L'' + \beta u_L &= h, \quad \text{in } \Omega_L, 
    \\
	\label{1D-bvp3} 
u_R (0) &= u_L (0), 
    \\
	\label{1D-bvp4} 
u_R' (0) - \alpha u_L' (0) &= \kappa u_R (0) + \lambda, 
    \\
	\label{1D-bvp5} 
u_L'(-1) &= 0 = u_R'(1).
    \end{align}

We define the solution $u_0$ to \qref{1D-bvp1}--\qref{1D-bvp5} over $\Omega$ as
\begin{align*}
	u_0(x) := \begin{cases}
		u_R(x),  &\text{if } x \in \Omega_R, \\
		u_L(x),  &\text{if } x \in \Omega_L.
	\end{cases}
\end{align*}
Let $u_\varepsilon$ be the solution to the corresponding diffuse domain approximation problem
\begin{align}
	\label{1D-diff-bvp1} -(D_\varepsilon u'_\varepsilon)' + c_\varepsilon u_\varepsilon + (\kappa u_\varepsilon + \lambda) \phi_\varepsilon' &= f_\varepsilon, \quad \text{in } \Omega, \\
	\label{1D-diff-bvp2} u'_\varepsilon(-1) &= 0 = u'_\varepsilon(1).
\end{align}	
Assuming that $u_0$ and $u_\varepsilon$ are globally continuous across $\Omega$, 
we will show that $u_\varepsilon$ asymptotically converges to $u_0$, using the method of matched asymptotic expansions.

To demonstrate asymptotic analysis, we examine the expansions of the diffuse domain solution $u_\varepsilon$ as series in powers of the interface thickness, $\varepsilon$, in two distinct regions: near $x=0$ (inner expansion) and far from $x=0$ (outer expansion). 
These inner and outer expansions are then matched in overlapping regions where both are valid (see Figure~\ref{domain-1D-regions}), 
providing boundary conditions for the outer variables.

\begin{figure}[htb!]
    \centering
  \begin{tikzpicture}[scale=1]

\draw[->] (-4,0) -- (4,0) node[right] {$x$};

\foreach \x in {-4,0,4} {
  \pgfmathtruncatemacro{\ticklabel}{\x/4}
  \draw (\x,0.05) -- (\x,-0.05) node[below] {$\ticklabel$};
}

\fill[lightgray] (-4,0) rectangle (-2,2);
\draw (-3,1) node {\parbox{3cm}{\centering $\text{Left}$ \\ $\text{outer region}$}};

\fill[pattern=north east lines] (-2,0) rectangle (-1,2);

\fill[gray] (-1,0) rectangle (1,2);
\draw (0,1) node {\parbox{3cm}{\centering $\text{Inner}$ \\ $\text{region}$}};

\fill[pattern=north east lines] (1,0) rectangle (2,2);

\fill[lightgray] (2,0) rectangle (4,2);
\draw (3,1) node {\parbox{3cm}{\centering $\text{Right}$ \\ $\text{outer region}$}};

\draw[->, ultra thick, bend left=20]  (-1.2,-1) to (-1.5,0);

\draw[->, ultra thick, bend left=-20]  (1.2,-1) to (1.5,0);
\draw (0,-1.5) node[above] {\small Matching regions};

\end{tikzpicture}
    \caption{A sketch of the regions used for the matched asymptotic expansions.}
	\label{domain-1D-regions}
\end{figure}

    \subsection{Matched Asymptotics for Corner Layer Problems}

The phenomenon that we will encounter herein is what is called a \emph{corner layer}~\cite{Miller2006}. To be more precise, $u_\varepsilon$ experiences a corner layer at $x = 0$. See, for example, Figure~\ref{fig:CornerLayer}. The idea of asymptotic analysis is to build a uniformly-valid composite solution, $u_{c,0}(x;\varepsilon)$ by combining the zeroth-order outer solution, $u_0$, and the zeroth-order inner solution, $U_0$, in a process called matching. Then, we can show that, formally, as $\varepsilon \to 0$, 
    \begin{equation}
\left\| u_\varepsilon(\, \cdot \, ) - u_{c,0}(\, \cdot \, ;\varepsilon) \right\|_{L^\infty(-1,1)} \le C_1\varepsilon,
    \label{eqn:comp-approx}
    \end{equation}
for some $C_1>0$. For corner layer phenomena, the zeroth-order approximation is continuous, that is,  $u_0\in C^0([-1,1])$, in contrast with the boundary layer case. Furthermore, for corner layer problems specifically, we typically expect that,  as $\varepsilon \to 0$,
    \[
\left\| u_0(\, \cdot \, ) - u_{c,0}(\, \cdot \, ;\varepsilon) \right\|_{L^\infty(-1,1)} \le C_2\varepsilon,
    \]
for some $C_2>0$. (Note that the last estimate is not generally true for traditional boundary layer problems.)  If this is the case, using the triangle inequality, the desired estimate follows: as $\varepsilon \to 0$,
    \begin{equation}
\left\| u_0(\, \cdot \, ) - u_\varepsilon(\, \cdot \, ) \right\|_{L^\infty(-1,1)} \le C_3\varepsilon,
    \label{eqn:basic-approx}
    \end{equation}
for some $C_3>0$, where, precisely, it will be shown that $u_0$ is the solution of \qref{1D-bvp1}--\qref{1D-bvp5} over $\Omega = (-1,1)$.

Finally, it is possible to argue that, in fact, as $\varepsilon \to 0$,
    \[
\left\| u_0(\, \cdot \, ) - u_\varepsilon(\, \cdot \, ) \right\|_{L^\infty(-1,1)} \le C_4\varepsilon^2,
    \]
for some $C_4>0$, only if $u_1\equiv 0$, where $u_1$ is the first-order correction term in the outer expansion. The arguments are formal and are omitted for the sake of brevity. Herein we will show definitively that $u_1\not\equiv 0$, so that asymptotic convergence is, at best, only first order for the two-sided problem. This fact will be reinforced by numerical experiments.

    \begin{figure}
        \centering
        \includegraphics[width=0.65\linewidth]{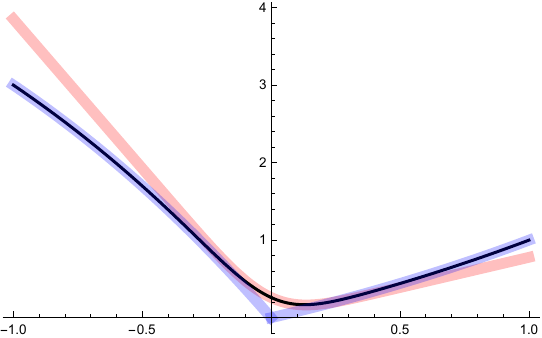}
        \caption{An example of a corner layer. The translucent blue curve is the zeroth-order outer approximation, $u_0$, and the translucent red curve represents the zeroth-order inner approximation, $U_0$. The solid black line is the diffuse-domain solution $u_\varepsilon(x)$. Not shown is the composite approximation, $u_{c,0}(x;\varepsilon)$, obtained by combining $u_0$ and $U_0$ via a process known as matching. See estimate \eqref{eqn:comp-approx} and associated discussion.}
        \label{fig:CornerLayer}
    \end{figure}

\subsection{Outer expansions}
We define $u_{\varepsilon, L} := \restr{u_\varepsilon}{\Omega_L}$ and $u_{\varepsilon, R} := \restr{u_\varepsilon}{\Omega_R}$.
The outer expansion for the function $u_{\varepsilon, L} (x)$ in the left outer region is
\begin{align}\label{outer-u-L}
	u_{\varepsilon, L} (x) = u_{L,0} (x) + \varepsilon u_{L,1} (x) + \varepsilon^2 u_{L,2} (x)+ \cdots 
\end{align}
Since $\phi_\varepsilon(x)$ and $\phi'_\varepsilon(x)$ exponentially converge to $0$ as $\varepsilon \to 0$ in the left outer region, 
by substituting the expansion \qref{outer-u-L} into Equation~\qref{1D-diff-bvp1}, 
we get
\begin{align}
	-\alpha (u''_{L,0} + \varepsilon u''_{L,1} + \varepsilon^2 u''_{L,2} + \cdots ) + \beta (u_{L,0} + \varepsilon u_{L,1} + \varepsilon^2 u_{L,2} + \cdots ) = h.
\end{align}
Combining the equation above with the boundary condition $u'_{\varepsilon, L} (-1) = 0$, we obtain:
\begin{align}
	\label{outer-u-L0} O(1): \; &-\alpha u''_{L,0} + \beta u_{L,0} = h, \quad u'_{L,0} (-1) = 0, \\	
	\label{outer-u-Lj} O(\varepsilon^j): \; &-\alpha u''_{L,j} + \beta u_{L,j} = 0, \quad u'_{L,j} (-1) = 0, \text{ for any } j = 1,2,\ldots
\end{align}

On the other hand, the outer expansion for the function $u_{\varepsilon, R} (x)$ in the right outer region is
\begin{align}\label{outer-u-R}
	u_{\varepsilon, R} (x) = u_{R,0} (x) + \varepsilon u_{R,1} (x) + \varepsilon^2 u_{R,2} (x)+ \cdots 
\end{align}
Since $\phi_\varepsilon(x)$ exponentially converges to $1$ and $\phi'_\varepsilon(x)$ exponentially converges to $0$ as $\varepsilon \to 0$ in the right outer region, 
by substituting the expansion \qref{outer-u-R} into Equation~\qref{1D-diff-bvp1}, 
we get
\begin{align}
	-(u''_{R,0} + \varepsilon u''_{R,1} + \varepsilon^2 u''_{R,2} + \cdots ) + \gamma (u_{R,0} + \varepsilon u_{R,1} + \varepsilon^2 u_{R,2} + \cdots ) = q.
\end{align}
Combining the equation above with the boundary condition $u'_{\varepsilon, R} (1) = 0$, we obtain:
\begin{align}
	\label{outer-u-R0} O(1): \; &-u''_{R,0} + \gamma u_{R,0} = q, \quad u'_{R,0} (1) = 0, \\	
	\label{outer-u-Rj} O(\varepsilon^j): \; &-u''_{R,j} + \gamma u_{R,j} = 0, \quad u'_{R,j} (1) = 0, \text{ for any } j = 1,2,\ldots 
\end{align}

\begin{rema}
    If the pair $(u_{R,0}, u_{L,0})$ satisfies \qref{1D-bvp1}--\qref{1D-bvp5}, then $u_\varepsilon$ converges asymptotically to $u_0$. 
    Furthermore, if $(u_{R,1}, u_{L,1}) \neq (0,0)$, the diffuse domain approximation is asymptotically first-order accurate. 
    However, if $(u_{R,1}, u_{L,1}) = (0,0)$, then the diffuse domain approximation is asymptotically second-order accurate. 
\end{rema}

\subsection{Supporting Lemmas}
Before analyzing the inner expansions, we prove the following technical lemmas.

\begin{lemma}\label{rational-lemma}
	If $F(z) = P(z) + o(1)$ and $G(z) = A + o(z^{-m})$ as $z \to \infty$ (or $z \to -\infty$), 
	where $P(z)$ is a polynomial, 
	$A \neq 0$ is a constant, 
	and $m > \deg(P)$ is an integer, then
	\begin{align*}
		\frac{F(z)}{G(z)} = \frac{P(z)}{A} + o(1) \quad  \text{as } z \to \infty \text{ (or  $z \to -\infty$)}.
	\end{align*}	
\end{lemma}

    \begin{proof}	
We will give the proof for the case $z \to \infty$. 
The proof for the case $z \to -\infty$ is similar. 
Assume that
\begin{align}
	P(z) = B_0 + B_1 z + B_2 z^2 + \cdots + B_p z^p,
\end{align}
where $0 \leq p < m$ is an integer, 
and $B_0, B_1,\ldots , B_p$ are constants. 
By the hypothesis, we have
\begin{align}
	F(z) &= P(z) + \tilde{F}(z), \; \text{ where } \lim_{z \to \infty} \tilde{F}(z) = 0,  \\
	G(z) &= A + \tilde{G}(z), \; \text{ where } \lim_{z \to \infty} z^m \tilde{G}(z) = 0.
\end{align}
Thus,
\begin{align}
	\frac{F(z)}{G(z)} - \frac{P(z)}{A} = \frac{P(z) + \tilde{F}(z)}{A + \tilde{G}(z)} - \frac{P(z)}{A}
	= \frac{A \tilde{F}(z) - P(z) \tilde{G}(z)}{A^2 + A \tilde{G}(z)}.
\end{align}
Since $m > 0$ and $\lim_{z \to \infty} z^m \tilde{G}(z) = 0$, then $\lim_{z \to \infty}  \tilde{G}(z)$ cannot equal $-A$. 
Hence, $\lim_{z \to \infty}  (A^2 + A \tilde{G}(z)) \neq 0$. 
On the other hand, 
\begin{align}
	P(z) \tilde{G}(z) = \sum_{j=0}^{p} B_j z^j \tilde{G}(z)
	= \sum_{j=0}^{p} \frac{B_j}{z^{m-j}} z^m \tilde{G}(z) \to 0 \quad \text{as } z \to \infty.
\end{align}
Therefore, we obtain
\begin{align}
	\lim_{z \to \infty} \left( \frac{F(z)}{G(z)} - \frac{P(z)}{A} \right) = 
	\lim_{z \to \infty} \frac{A \tilde{F}(z) - P(z) \tilde{G}(z)}{A^2 + A \tilde{G}(z)} = 0,
\end{align}
which implies that
\begin{align*}
	\frac{F(z)}{G(z)} = \frac{P(z)}{A} + o(1) \quad \text{as } z \to \infty.
\end{align*}	
\end{proof}

\begin{lemma}\label{integral-lemma}
	If $F(z)$ and $G(z)$ are continuous functions on $(-\infty, \infty)$, and $F(z) = G(z) + o(1)$ as $z \to \infty$ (or $z \to -\infty$), then
	\begin{align*}
		\int_0^z F(s)ds = \int_0^z G(s)ds + o(z) \quad  \text{as } z \to \infty \text{ (or  $z \to -\infty$)}.
	\end{align*}
\end{lemma}

    \begin{proof}		
Again, we will give the proof for the case $z \to \infty$, 
and the proof for the case $z \to -\infty$ is similar. 
Pick an arbitrary number $\xi > 0$. 
By the hypothesis, we have 
\begin{align}
	\lim_{z \to \infty} (F(z) - G(z)) = 0.
\end{align}
Hence, there exists a number $z_0 > 0$, depending on $\xi$, such that, for any $z > z_0$,
\begin{align}
	|F(z) - G(z)| < \frac{\xi}{2}.
\end{align}
Therefore, for any $z > z_0 > 0$, we have
\begin{align}
	\left| \frac{1}{z} \int_0^z (F(s) - G(s)) \;ds \right| &\leq \frac{1}{z} \int_0^z |F(s) - G(s)| \;ds \nonumber \\
	&\leq \frac{1}{z} \int_0^{z_0} |F(s) - G(s)| \;ds + \frac{1}{z} \int_{z_0}^z |F(s) - G(s)|\; ds \nonumber \\
	&\leq \frac{M_0}{z} + \frac{\xi}{2z} (z - z_0)   \nonumber \\
	&\leq \frac{\xi}{2} + \frac{1}{z} \left( M_0 - \frac{\xi}{2} z_0 \right),
\end{align}
where $M_0 := \int_0^{z_0} |F(s) - G(s)| \;ds < \infty$. 
If $M_0 - \frac{\xi}{2} z_0 \leq 0$, then 
\begin{align}
	\left| \frac{1}{z} \int_0^z (F(s) - G(s)) \;ds \right| \leq \frac{\xi}{2} < \xi,
\end{align}
for any $z > z_0$.  If $M_0 - \frac{\xi}{2} z_0 > 0$, define $z_1 := \max \left\{ z_0, \frac{2}{\xi} \left( M_0 - \frac{\xi}{2} z_0 \right) \right\} > 0$. Thus, for any $z > z_1$, we have
\begin{align}
	\left| \frac{1}{z} \int_0^z (F(s) - G(s))\; ds \right| &\leq  \frac{\xi}{2} + \frac{1}{z} \left( M_0 - \frac{\xi}{2} z_0 \right) \nonumber \\
	&\leq \frac{\xi}{2} + \frac{\xi}{2 \left( M_0 - \frac{\xi}{2} z_0 \right)} \left( M_0 - \frac{\xi}{2} z_0 \right) \nonumber \\
	&\leq \xi.
\end{align}
In both cases, we have 
\begin{align}
	\left| \frac{1}{z} \int_0^z (F(s) - G(s))\; ds \right| \leq \xi,
\end{align}
for $z > 0$ sufficiently large. 
Therefore, 
\begin{align}
	\lim_{z \to \infty}  \frac{1}{z} \int_0^z (F(s) - G(s)) \;ds = 0,
\end{align}
which implies that
\begin{align*}
	\int_0^z F(s)ds = \int_0^z G(s)ds + o(z) \quad \text{as } z \to \infty.
\end{align*}
\end{proof}

\subsection{Inner expansions}
We introduce a stretched variable
\begin{align*}
	z = \frac{x}{\varepsilon},
\end{align*}
and transform the derivatives as
\begin{align*}
	\frac{d}{dx} = \frac{1}{\varepsilon} \frac{d}{dz}, \quad
	\frac{d^2}{dx^2} = \frac{1}{\varepsilon^2} \frac{d^2}{dz^2}.
\end{align*}
Furthermore, we define the functions
\begin{align}
	U(z, \varepsilon) &:= u_\varepsilon(\varepsilon z), \\
	\label{phi(z)} \phi(z) &:= \phi_\varepsilon(\varepsilon z) = \frac{1}{2}[1+\tanh(z)], \\
	\label{D(z)} D(z) &:= D_\varepsilon(\varepsilon z) = \alpha + (1 - \alpha)\phi(z), \\
	\label{c(z)} c(z) &:= c_\varepsilon(\varepsilon z) = \beta + (\gamma - \beta)\phi(z), \\
	F(z, \varepsilon) &:= f_\varepsilon(\varepsilon z).
\end{align}
Then, Equation~\qref{1D-diff-bvp1} becomes
\begin{align}
    -\frac{1}{\varepsilon^2}\frac{d}{dz}\left(D(z) \frac{d}{dz}U(z, \varepsilon)\right) + c(z) U(z, \varepsilon) + \frac{1}{\varepsilon}(\kappa U(z, \varepsilon) + \lambda) \phi'(z) = F(z, \varepsilon),
\end{align}
which is equivalent to
\begin{align}\label{1D-diff-bvp1-inner}
	-\frac{d}{dz}\left(D(z) \frac{d}{dz}U(z, \varepsilon)\right) + c(z) U(z, \varepsilon) \varepsilon^2 + (\kappa U(z, \varepsilon) + \lambda) \phi'(z) \varepsilon = F(z, \varepsilon) \varepsilon^2.
\end{align}
Since $h(x)$ and $q(x)$ are analytic functions, we can expand them as
\begin{align}
	h(x) &= h_0 + h_1 x + h_2 x^2 + \cdots , \\
	q(x) &= q_0 + q_1 x + q_2 x^2 + \cdots ,	
\end{align}
where $h_j$ and $q_j$ are constants, for $j = 0,1,2,\ldots$ 
Hence,
\begin{align}\label{inner-F}
	F(z, \varepsilon) = f_\varepsilon(\varepsilon z) 
	&= h(\varepsilon z) + [q(\varepsilon z) - h(\varepsilon z)]\phi(z) \nonumber \\
	&= \sum_{j=1}^\infty [h_j + (q_j - h_j)\phi(z)]z^j \varepsilon^j.
\end{align}
The inner expansion for the function $U(z, \varepsilon)$ in the inner region is
\begin{align}\label{inner-U}
	U(z, \varepsilon) = U_0 (z) + \varepsilon U_1 (z) + \varepsilon^2 U_2 (z)+ \cdots 
\end{align}
We assume that $U_0(z), U_1(z), U_2(z), \ldots$ are continuous functions on $(-\infty, \infty)$, 
which is consistent with the assumption that $u_\varepsilon(x)$ is globally continuous across $\Omega$. 
Substituting the expansions \qref{inner-F} and \qref{inner-U} into Equation~\qref{1D-diff-bvp1-inner}, 
we get
\begin{align}
	&-\sum_{j=1}^\infty (D(z) U_j'(z))' \varepsilon^j + \sum_{j=1}^\infty c(z) U_j(z) \varepsilon^{j+2} + \left( \sum_{j=1}^\infty \kappa U_j(z) \varepsilon^{j+1} + \lambda\varepsilon \right) \phi'(z) \nonumber \\
	&= \sum_{j=1}^\infty [h_j + (q_j - h_j)\phi(z)]z^j \varepsilon^{j+2}.
\end{align}
Hence, we obtain
    \begin{align}
	\label{inner-U-R0} 
O(1): \; &-(D(z) U_0'(z))' = 0, 
    \\	
	\label{inner-U-R1} 
O(\varepsilon): \; &-(D(z) U_1'(z))' + (\kappa U_0(z) + \lambda) \phi'(z) = 0,     \\	
	\label{inner-U-Rj} 
O(\varepsilon^j): \; &-(D(z) U_j'(z))' + c(z) U_{j-2}(z) + \kappa U_{j-1}(z)\phi'(z) \nonumber \\
&=  [h_j + (q_j - h_j)\phi(z)]z^{j-2},\quad j = 2,3,\ldots 
    \end{align}

\subsubsection{Equation~\qref{1D-diff-bvp1-inner} at $O(1)$}
From \qref{inner-U-R0}, we get 
\begin{align}
	(D(z) U_0'(z))' = 0. 
\end{align}
Then, there exists a constant $C_{0,1}$ such that
\begin{align}
	D(z) U_0'(z) = C_{0,1}. 
\end{align}
There exists another constant $C_{0,0}$ such that
\begin{align}
	U_0(z) = C_{0,1} \int_0^z \frac{1}{D(s)} \;ds + C_{0,0}. 
\end{align}
We define
\begin{align}
	I_0 (z) := \int_0^z  \frac{1}{D(s)}\; ds.
\end{align}
Using the definition of $D(z)$, given by
\begin{align}
	D(z) = \alpha + (1 - \alpha)\phi(z) = \frac{e^z + \alpha e^{-z}}{e^z + e^{-z}},
\end{align}
we obtain
\begin{align}\label{I-0}
	I_0 (z) &= \int_0^z \frac{e^s + e^{-s}}{e^s + \alpha e^{-s}}\; ds \nonumber \\
	& = \frac{(\alpha - 1) \log(\alpha + e^{2z}) + 2z}{2\alpha} - \frac{(\alpha - 1) \log(\alpha + 1)}{2\alpha}.
\end{align}
We will use the Taylor series to expand $\log(\alpha + e^{2z})$.

\begin{rema}
    The Taylor series expansion of $\log(1+y)$, for $|y| < 1$, is given by:    
    \begin{align}\label{Taylor-series-1}
	   \log(1+y) = y - \frac{1}{2} y^2 + \frac{1}{3} y^3 - \frac{1}{4} y^4 + \cdots 
    \end{align}
\end{rema}

For $z > 0$ sufficiently large, applying \qref{Taylor-series-1} with $y = \alpha e^{-2z}$, we obtain
    \begin{align*}
\log(\alpha + e^{2z}) & = \log[e^{2z}(\alpha e^{-2z} + 1)]  
    \nonumber  
    \\
& = 2z + \log(\alpha e^{-2z} + 1)  
    \nonumber  
    \\
& = 2z + \alpha e^{-2z} - \frac{1}{2} \left(\alpha e^{-2z}\right)^2 + \frac{1}{3} \left(\alpha e^{-2z}\right)^3 - \frac{1}{4} \left(\alpha e^{-2z}\right)^4 + \cdots ,
    \end{align*}
which implies that
    \begin{align}
    \label{log-alpha-right}
\log(\alpha + e^{2z}) = 2z + o(1) \quad \text{as } z \to \infty.
    \end{align}
Substituting the equation above into \qref{I-0}, we get
\begin{align}\label{right-I-0}
	I_0 (z) &= \frac{(\alpha - 1) (2z + o(1)) + 2z}{2\alpha} - \frac{(\alpha - 1) \log(\alpha + 1)}{2\alpha} \nonumber \\
	& = z - \frac{(\alpha - 1) \log(\alpha + 1)}{2\alpha} + o(1) \quad \text{as } z \to \infty.
\end{align}
For $z < 0$ with $|z|$ sufficiently large, applying \qref{Taylor-series-1} with $y = e^{2z}/\alpha$, we obtain
\begin{align*}
	\log(\alpha + e^{2z}) &= \log \left[\alpha \left(1 + \frac{e^{2z}}{\alpha} \right)\right]  \nonumber  \\
	&= \log \alpha + \log \left(1 + \frac{e^{2z}}{\alpha} \right)  \nonumber  \\
	&= \log \alpha + \frac{e^{2z}}{\alpha} - \frac{1}{2} \left( \frac{e^{2z}}{\alpha} \right)^2 + \frac{1}{3} \left( \frac{e^{2z}}{\alpha} \right)^3 - \frac{1}{4} \left( \frac{e^{2z}}{\alpha} \right)^4 + \cdots ,
\end{align*}
which implies that
\begin{align}\label{log-alpha-left}
	\log(\alpha + e^{2z}) = \log \alpha + o(1) \quad \text{as } z \to -\infty.
\end{align}
Substituting the equation above into \qref{I-0}, we get
\begin{align}\label{left-I-0}
	I_0 (z) &= \frac{(\alpha - 1) (\log \alpha + o(1)) + 2z}{2\alpha} - \frac{(\alpha - 1) \log(\alpha + 1)}{2\alpha} \nonumber \\
	& = \frac{1}{\alpha} z + \frac{\alpha - 1}{2\alpha} \log \left(\frac{\alpha}{\alpha+1}\right) + o(1) \quad \text{as } z \to -\infty.
\end{align}
Since $U_0(z) = C_{0,1} I_0(z) + C_{0,0}$, from \qref{right-I-0} and \qref{left-I-0}, we obtain
\begin{align}
	U_0 (z) = 
	\begin{cases} 
		\displaystyle C_{0,1}z - C_{0,1} \frac{(\alpha - 1) \log(\alpha + 1)}{2\alpha} + C_{0,0} + o(1) \quad \text{as } z \to \infty, \\
		\\
		\displaystyle \frac{C_{0,1}}{\alpha}z - C_{0,1} \frac{\alpha - 1}{2\alpha} \log \left(\frac{\alpha}{\alpha+1}\right) + C_{0,0} + o(1) \quad \text{as } z \to -\infty.
	\end{cases}
\end{align}
Then, using the matching conditions
\begin{align}
	U_0(z) + o(1) &= u_{R,0}(0) \quad \text{as } z \to \infty, \\
	U_0(z) + o(1) &= u_{L,0}(0) \quad \text{as } z \to -\infty,
\end{align}
we conclude that
\begin{align}\label{u-LR0-at-0}
	C_{0,1} = 0 \quad \text{and} \quad  u_{R,0}(0) = C_{0,0} = u_{L,0}(0).
\end{align}
As a consequence, the zeroth-order inner approximation is a constant function,
    \[
U_0(z) \equiv u_{R,0}(0) = u_{L,0}(0),
    \]
and, using the standard construction~\cite{Miller2006}, the uniformly-valid, zeroth-order composite asymptotic approximation is simply
    \begin{align*}
u_{c,0}(x;\varepsilon) := 
    \begin{cases}
u_{R,0}(x),  &\text{if } x \in (-1,0], 
    \\
u_{L,0}(x),  &\text{if } x \in (0,1),
	\end{cases}
    \end{align*}
which is generally only a $C^0([-1,1])$ function.

Looking ahead,  we next intend to show that $u_{R,0}$ and $u_{L,0}$ satisfy the Robin-type boundary condition in \eqref{1D-bvp4}. Going a little further in the analysis, we will observe that the composite approximation, $u_{c,0}$ is at best only first-order, by observing that the outer first-order corrections, $u_{R,1}$ and $u_{L,1}$, are not identically zero, in the general setting.

\subsubsection{Equation~\qref{1D-diff-bvp1-inner} at $O(\varepsilon)$}
Since $U_0(z) \equiv C_{0,0}$, from \qref{inner-U-R1} we have 
\begin{align}
	-(D(z) U_1'(z))' + (\kappa C_{0,0} + \lambda) \phi'(z) = 0. 
\end{align}
Then, here exists a constant $C_{1,1}$ such that
\begin{align}
	-D(z) U_1'(z) + (\kappa C_{0,0} + \lambda) \phi(z) = C_{1,1}. 
\end{align}
There exists another constant $C_{1,0}$ such that
\begin{align}
	U_1(z) =  (\kappa C_{0,0} + \lambda) \int_0^z \frac{\phi(s)}{D(s)} ds + C_{1,1} \int_0^z \frac{1}{D(s)} ds + C_{1,0}. 
\end{align}
We define
\begin{align}
	I_1 (z) := \int_0^z  \frac{\phi(s)}{D(s)} \;ds.
\end{align}
Using the definitions of $\phi(z)$ and $D(z)$, given by
\begin{align}
	\phi(z) = \frac{e^z}{e^z + e^{-z}} \quad \text{and} \quad D(z) = \frac{e^z + \alpha e^{-z}}{e^z + e^{-z}},
\end{align}
we obtain
\begin{align}
	I_1 (z) &= \int_0^z \frac{e^s}{e^s + \alpha e^{-s}} \;ds \nonumber \\
	& = \frac{1}{2} \log(\alpha + e^{2z}) - \frac{1}{2}\log(\alpha + 1).
\end{align}
From \qref{log-alpha-right} and \qref{log-alpha-left}, we have
\begin{align}
	\log(\alpha + e^{2z}) &= 2z + o(1) \quad \text{as } z \to \infty, \\
	\log(\alpha + e^{2z}) &= \log \alpha + o(1) \quad \text{as } z \to -\infty.
\end{align}
Therefore,
\begin{align}\label{I-1}
	I_1 (z) &= 
	\begin{cases} 
	\displaystyle \frac{1}{2} (2z + o(1)) - \frac{1}{2} \log(\alpha + 1) \quad \text{as } z \to \infty, \\
	\\
	\displaystyle \frac{1}{2} (\log \alpha + o(1)) - \frac{1}{2} \log(\alpha + 1) \quad \text{as } z \to -\infty, \nonumber
	\end{cases}
	\\ &= 
	\begin{cases} 
		\displaystyle z - \frac{1}{2} \log(\alpha + 1) + o(1) \quad \text{as } z \to \infty, \\
		\\
		\displaystyle \frac{1}{2} \log \left(\frac{\alpha}{\alpha + 1}\right) + o(1) \quad \text{as } z \to -\infty.
	\end{cases}
\end{align}
Since $U_1(z) = (\kappa C_{0,0} + \lambda) I_1(z) + C_{1,1} I_0(z) + C_{1,0}$, we obtain
\begin{align}\label{U1-right}
	U_1 (z) =& \; (\kappa C_{0,0} + \lambda + C_{1,1})z + \left( -\frac{1}{2}(\kappa C_{0,0} + \lambda) - C_{1,1}\frac{\alpha - 1}{2\alpha} \right) \log(\alpha + 1)  \nonumber \\
    &+ C_{1,0} + o(1) \quad \text{as } z \to \infty,
\end{align}
and
\begin{align}\label{U1-left}
	U_1 (z) =& \; \frac{C_{1,1}}{\alpha}z + \left( \frac{1}{2}(\kappa C_{0,0} + \lambda) + C_{1,1}\frac{\alpha - 1}{2\alpha} \right) \log \left( \frac{\alpha}{\alpha + 1} \right)   \nonumber \\
    &+ C_{1,0} + o(1) \quad \text{as } z \to -\infty,
\end{align}
Then, using the matching conditions
\begin{align}
	\label{match-U1-R} U_1(z) + o(1) &= u_{R,1}(0) + u'_{R,0}(0)z \quad \text{as } z \to \infty, \\
	\label{match-U1-L} U_1(z) + o(1) &= u_{L,1}(0) + u'_{L,0}(0)z \quad \text{as } z \to -\infty,
\end{align}
we obtain 
\begin{align}
	\kappa C_{0,0} + \lambda + C_{1,1} = u'_{R,0}(0) \quad \text{and} \quad  \frac{C_{1,1}}{\alpha} = u'_{L,0}(0).
\end{align}
Combining the two equations above with $C_{0,0} = u_{R,0}(0)$, we get
\begin{align}\label{u'R0(0)-au'R0(0)}
	u_{R,0}' (0) - \alpha u_{L,0}' (0) = \kappa u_{R,0} (0) + \lambda.
\end{align}
From \qref{outer-u-L0}, \qref{outer-u-R0}, \qref{u-LR0-at-0} and \qref{u'R0(0)-au'R0(0)}, we see that $u_{R,0}(x)$ and $u_{L,0}(x)$ are determined by the following boundary value problem
\begin{align}
	-u_{R,0}'' + \gamma u_{R,0} &= q, \quad \text{in } \Omega_R, \\
	-\alpha u_{L,0}'' + \beta u_{L,0} &= h, \quad \text{in } \Omega_L, \\
	u_{R,0}(0) &= u_{L,0}(0), \\
	u_{R,0}' (0) - \alpha u_{L,0}' (0) &= \kappa u_{R,0} (0) + \lambda, \\
	u_{L,0}'(-1) &= 0 = u_{R,0}'(1),
\end{align}
which means, $(u_{R,0}, u_{L,0})$ is a pair of solutions to the two-sided problem \qref{1D-bvp1}--\qref{1D-bvp5}. 
Therefore, the solution $u_\varepsilon$ of the diffuse domain approximation problem \qref{1D-diff-bvp1}--\qref{1D-diff-bvp2} converges asymptotically to the solution $u_0$ of the two-sided problem \qref{1D-bvp1}--\qref{1D-bvp5}, 
since this problem has a unique solution $u_0$. 

Moreover, using the matching conditions \qref{match-U1-R} and \qref{match-U1-L} again, from \qref{U1-right} and \qref{U1-left}, we obtain
\begin{align}
	\label{uR1(0)} u_{R,1}(0) &= \left( -\frac{1}{2}(\kappa C_{0,0} + \lambda) - C_{1,1}\frac{\alpha - 1}{2\alpha} \right) \log(\alpha + 1) + C_{1,0}, \\
	\label{uL1(0)} u_{L,1}(0) &=  \left( \frac{1}{2}(\kappa C_{0,0} + \lambda) + C_{1,1}\frac{\alpha - 1}{2\alpha} \right) \log \left( \frac{\alpha}{\alpha + 1} \right) + C_{1,0}.
\end{align}
Hence,
\begin{align}\label{uL1(0)-uR1(0)}
	u_{L,1}(0) - u_{R,1}(0) 
	= \left( \frac{1}{2}(\kappa C_{0,0} + \lambda) + C_{1,1}\frac{\alpha - 1}{2\alpha} \right) \log\alpha.
\end{align}	
Substituting  $C_{0,0} = u_{R,0}(0)$ and  $C_{1,1} = \alpha u'_{L,0}(0)$ into Equation \qref{uL1(0)-uR1(0)}, and simplifying the resulting expression, we obtain
\begin{align}\label{uL1(0)-uR1(0)-b}
	u_{L,1}(0) - u_{R,1}(0) = \frac{\log\alpha}{2}[u_{R,0}' (0) - u_{L,0}' (0)].
\end{align}

\subsubsection{Equation~\qref{1D-diff-bvp1-inner} at $O(\varepsilon^2)$}
Since $U_0(z) = C_{0,0}$, from \qref{inner-U-Rj} we have 
\begin{align}
	-(D(z) U_2'(z))' + C_{0,0}c(z) + \kappa U_1(z) \phi'(z) = h_0 + (q_0 - h_0) \phi(z). 
\end{align}
Then, there exists a constant $C_{2,1}$ such that
\begin{align}\label{D-U'2-a}
	D(z) U_2'(z) =& \; C_{0,0} \int_0^z c(s)\;ds + \kappa \int_0^z U_1(s) \phi'(s)\; ds - h_0 z   \nonumber \\
    &+ (h_0 - q_0) \int_0^z \phi(s)\;ds + C_{2,1}. 
\end{align}
We define
\begin{align}
	I_2 (z) := \int_0^z \phi(s)\;ds.
\end{align}
Using the definition of $\phi(z)$, given by
\begin{align}
	\phi(z) = \frac{e^z}{e^z + e^{-z}},
\end{align}
we obtain
\begin{align}\label{I-2}
	I_2 (z) &= \int_0^z \frac{e^s}{e^s + e^{-s}}\; ds \nonumber \\
	& = \frac{1}{2} \log(1 + e^{2z}) - \frac{1}{2}\log 2.
\end{align}
Since $U_1(z) = (\kappa C_{0,0} + \lambda) I_1(z) + C_{1,1} I_0(z) + C_{1,0}$ and
\begin{align}\label{I0-I1-relation}
	I_0(z) = \frac{\alpha - 1}{\alpha} I_1(z) + \frac{z}{\alpha},
\end{align} 
we have
\begin{align}\label{U'1-phi}
	\int_0^z U_1(s) \phi'(s)\;ds =& \left( \kappa C_{0,0} + \lambda + \frac{\alpha - 1}{\alpha} \right)\int_0^z I_1(s) \phi'(s)\;ds   \nonumber \\
    &+ \frac{C_{1,1}}{\alpha} \int_0^z \phi'(s)s \;ds + C_{1,0} \int_0^z \phi'(s) \;ds \nonumber \\
	=& \left( \kappa C_{0,0} + \lambda + \frac{\alpha - 1}{\alpha} \right)\int_0^z I_1(s) \phi'(s)\;ds   \nonumber \\
    &+ \frac{C_{1,1}}{\alpha} \int_0^z \phi'(s)s \;ds + C_{1,0} \left( \phi(z) - \frac{1}{2} \right).
\end{align}
Substituting \qref{c(z)} and \qref{U'1-phi} into \qref{D-U'2-a}, we obtain
\begin{align}\label{D-U'2}
	D(z) U_2'(z)
	=& \; C_{0,0}[\beta z + (\gamma - \beta)I_2(z)] - h_0 z + (h_0 - q_0)I_2(z) \nonumber \\
	 &+ \left[ (\kappa^2 C_{0,0} +\kappa\lambda) + \frac{\alpha - 1}{\alpha}\kappa C_{1,1} \right] \int_0^z I_1(s) \phi'(s)\;ds  \nonumber \\
	 & + \frac{\kappa C_{1,1}}{\alpha} \int_0^z \phi'(s)s\; ds + \kappa C_{1,0} \phi(z) - \frac{\kappa C_{1,0}}{2} + C_{2,1}. 
\end{align}

\

\noindent $\bullet$ \textbf{Case 1: $\alpha \neq 1$.} 

Using integration by parts, we get
\begin{align}\label{s-phi's}
	\int_0^z \phi'(s)s \;ds &= \phi(s)s \bigg\rvert_0^z - \int_0^z \phi(s) \;ds  \nonumber \\
    &= \phi(z)z - I_2(z),
\end{align}
and
\begin{align}
	\int_0^z I_1(s)\phi'(s) \;ds &= I_1(s)\phi(s) \bigg\rvert_0^z - \int_0^z \phi(s) I'_1(s)\; ds \nonumber \\
	&= I_1(z)\phi(z) -\int_0^z \phi(s) \frac{\phi(s)}{D(s)}\;ds.
\end{align}
Since $D(s) = \alpha + (1 - \alpha)\phi(s)$, then
\begin{align}\label{phi/D}
	\frac{\phi(s)}{D(s)} = \frac{1}{1 - \alpha} - \frac{\alpha}{1 - \alpha} \frac{1}{D(s)}.
\end{align}
Thus,
\begin{align}
	\phi(s) \frac{\phi(s)}{D(s)} &= \frac{1}{1 - \alpha}\phi(s) - \frac{\alpha}{1 - \alpha} \frac{\phi(s)}{D(s)}  \nonumber \\
	&= \frac{1}{1 - \alpha}\phi(s) - \frac{\alpha}{1 - \alpha} \left( \frac{1}{1 - \alpha} - \frac{\alpha}{1 - \alpha} \frac{1}{D(s)} \right) \nonumber \\
	&= \frac{1}{1 - \alpha}\phi(s) - \frac{\alpha}{(1 - \alpha)^2} + \frac{\alpha^2}{(1 - \alpha)^2} \frac{1}{D(s)}.
\end{align}
Combining the equation above with \qref{I0-I1-relation}, we get
\begin{align}
	\int_0^z \phi(s) \frac{\phi(s)}{D(s)}\;ds 
	&= \frac{1}{1 - \alpha} I_2(z) - \frac{\alpha}{(1 - \alpha)^2} z + \frac{\alpha^2}{(1 - \alpha)^2} I_0(z) \nonumber \\
	&= \frac{1}{1 - \alpha} I_2(z) - \frac{\alpha}{(1 - \alpha)^2} z + \frac{\alpha^2}{(1 - \alpha)^2} \left(  \frac{\alpha - 1}{\alpha} I_1(z) + \frac{z}{\alpha} \right) \nonumber \\
	&= \frac{-1}{\alpha - 1} I_2(z) + \frac{\alpha}{(1 - \alpha)^2} I_1(z).
\end{align}
Therefore,
\begin{align}\label{I'1-phi}
	\int_0^z I_1(s)\phi'(s) \;ds = I_1(z)\phi(z) + \frac{1}{\alpha - 1} I_2(z) - \frac{\alpha}{(1 - \alpha)^2} I_1(z).
\end{align}
Substituting \qref{s-phi's} and \qref{I'1-phi} into \qref{D-U'2}, and simplifying the resulting expression, we get
\begin{align}\label{U'_2-a}
	U_2'(z) =& \; (C_{0,0}\beta - h_0)\frac{z}{D(z)} + \left( C_{0,0}(\gamma - \beta) + h_0 - q_0 + \frac{\kappa^2 C_{0,0} + \kappa\lambda}{\alpha - 1} \right)\frac{I_2(z)}{D(z)}  \nonumber \\
	&+ \left( \kappa^2 C_{0,0} + \kappa\lambda + \frac{\alpha - 1}{\alpha}\kappa C_{1,1} \right) \frac{I_1(z)\phi(z)}{D(z)}  \nonumber \\
	&+ \left( -(\kappa^2 C_{0,0} + \kappa\lambda)\frac{\alpha}{\alpha - 1} + \kappa C_{1,1} \right) \frac{I_1(z)}{D(z)} + \frac{\kappa C_{1,1}}{\alpha}\frac{\phi(z)z}{D(z)}  \nonumber \\
	& + \kappa C_{1,0}\frac{\phi(z)}{D(z)} + \left( \frac{-\kappa C_{1,0}}{2} + C_{2,1} \right)\frac{1}{D(z)}.
\end{align}
Using \qref{phi/D}, we obtain
\begin{align}
	\frac{I_1(z)\phi(z)}{D(z)} &= \frac{-I_1(z)}{\alpha - 1} + \frac{\alpha}{\alpha - 1} \frac{I_1(z)}{D(s)}, \\
	\frac{\phi(z)z}{D(z)} &= \frac{-z}{\alpha - 1} + \frac{\alpha}{\alpha - 1} \frac{z}{D(s)}.
\end{align}
Substituting the two equations above into \qref{U'_2-a} and simplifying the resulting expression, we get
\begin{align}\label{U'_2-b}
	U_2'(z) =& \left( C_{0,0}\beta - h_0 + \frac{\kappa C_{1,1}}{\alpha - 1}\right)\frac{z}{D(z)} + \left( \frac{-(\kappa^2 C_{0,0} + \kappa\lambda)}{\alpha - 1} + \frac{-\kappa C_{1,1}}{\alpha} \right) I_1(z)   \nonumber \\
	&+ \left( C_{0,0}(\gamma - \beta) + h_0 - q_0 + \frac{\kappa^2 C_{0,0} + \kappa\lambda}{\alpha - 1} \right)\frac{I_2(z)}{D(z)}  \nonumber \\
	&+ \frac{-\kappa C_{1,1}}{\alpha(\alpha - 1)} z + \frac{-\kappa C_{1,0}}{\alpha - 1} + \left( \frac{\alpha + 1}{2(\alpha - 1)}\kappa C_{1,0} + C_{2,1} \right)\frac{1}{D(z)}.
\end{align}

We examine the asymptotic limit of $U_2'(z)$ as $z \to \infty$ first.  
From \qref{I-1}, we have
\begin{align}\label{I1-right}
	I_1 (z) = z - \frac{1}{2} \log(\alpha + 1) + o(1) \quad \text{as } z \to \infty.
\end{align}
Since $z + ze^{-2z} = z + o(1)$, $1 + e^{-2z} = 1 + o(1)$, and $1+ \alpha e^{-2z} = 1 + o(z^{-10})$ as $z \to \infty$, 
applying Lemma~\ref{rational-lemma}, we get
\begin{align}
	\label{z/D-right} \frac{z}{D(z)} &= \frac{z + ze^{-2z}}{1+ \alpha e^{-2z}} = z + o(1) \quad \text{as } z \to \infty, \\
	\label{1/D-right} \frac{1}{D(z)} &= \frac{1 + e^{-2z}}{1+ \alpha e^{-2z}} = 1 + o(1) \quad \text{as } z \to \infty.
\end{align}
On the other hand,
\begin{align}
	\frac{I_2(z)}{D(z)} &= \frac{(\log(1 + e^{2z}) - \log2) (1 + e^{-2z})}{2(1 + \alpha e^{-2z})}.
\end{align}
For $z > 0$ sufficiently large, applying \qref{Taylor-series-1} with $y = e^{-2z}$, we obtain
\begin{align}\label{log1-expan-right}
	\log(1 + e^{2z}) &= \log[e^{2z}(e^{-2z} + 1)]  \nonumber  \\
	&= 2z + \log(e^{-2z} + 1)  \nonumber  \\
	&= 2z + \left( e^{-2z} - \frac{1}{2} (e^{-2z})^2 + \cdots  \right).
\end{align}
Hence,
\begin{align}
	(\log(1 + e^{2z}) - \log2) (1 + e^{-2z})  
	=& \log(1 + e^{2z}) + e^{-2z}\log(1 + e^{2z}) - \log2 - e^{-2z}\log2  \nonumber  \\
	=& \; 2z + \left( e^{-2z} - \frac{1}{2} (e^{-2z})^2 + \cdots  \right)  \nonumber  \\
	&+ 2z e^{-2z} + e^{-2z}\left( e^{-2z} - \frac{1}{2} (e^{-2z})^2 + \cdots  \right)  \nonumber  \\
    &- \log2 - e^{-2z}\log2, 
\end{align}
which implies that
\begin{align}
	(\log(1 + e^{2z}) - \log2) (1 + e^{-2z}) 
	= 2z - \log2  + o(1) \quad \text{as } z \to \infty.
\end{align}
Applying Lemma~\ref{rational-lemma} again, we get
\begin{align}\label{I2/D-right}
	\frac{I_2(z)}{D(z)} = z - \frac{\log2}{2}  + o(1) \quad \text{as } z \to \infty.
\end{align}
Substituting \qref{I1-right}, \qref{z/D-right}, \qref{1/D-right} and \qref{I2/D-right} into \qref{U'_2-b}, and simplifying the resulting expression, we obtain
\begin{align}
	U_2'(z) =& \; (C_{0,0} \gamma - q_0)z 
	+ \left( \frac{\kappa^2 C_{0,0} + \kappa\lambda}{\alpha - 1} + \frac{\kappa C_{1,1}}{\alpha}\right) \frac{\log(\alpha + 1)}{2} + \frac{\kappa C_{1,0}}{2}  \nonumber \\ 
	& + C_{2,1} - \left( C_{0,0}(\gamma - \beta) + h_0 - q_0 + \frac{\kappa^2 C_{0,0} + \kappa\lambda}{\alpha - 1} \right) \frac{\log2}{2}   \nonumber  \\
    & + o(1) \quad \text{as } z \to \infty.
\end{align}
Applying Lemma~\ref{integral-lemma}, we get
\begin{align}
	U_2(z) =& \; \frac{C_{0,0} \gamma - q_0}{2} z^2 
	+ \biggl[ \left( \frac{\kappa^2 C_{0,0} + \kappa\lambda}{\alpha - 1} + \frac{\kappa C_{1,1}}{\alpha}\right) \frac{\log(\alpha + 1)}{2} + \frac{\kappa C_{1,0}}{2}  \nonumber \\ 
	& + C_{2,1} - \left( C_{0,0}(\gamma - \beta) + h_0 - q_0 + \frac{\kappa^2 C_{0,0} + \kappa\lambda}{\alpha - 1} \right) \frac{\log2}{2} \biggr]z   \nonumber  \\
    & + o(z) \quad \text{as } z \to \infty.
\end{align}
Then, using the matching condition
\begin{align}
	\label{match-U2-R} U_2(z) + o(1) &= u_{R,2}(0) + u'_{R,1}(0)z + \frac{1}{2} u''_{R,0}(0)z^2 \quad \text{as } z \to \infty,
\end{align}
we obtain 
\begin{align}\label{u'-R1}
	u'_{R,1}(0) =& \left( \frac{\kappa^2 C_{0,0} + \kappa\lambda}{\alpha - 1} + \frac{\kappa C_{1,1}}{\alpha}\right) \frac{\log(\alpha + 1)}{2} + \frac{\kappa C_{1,0}}{2} + C_{2,1}  \nonumber \\ 
	&  - \left( C_{0,0}(\gamma - \beta) + h_0 - q_0 + \frac{\kappa^2 C_{0,0} + \kappa\lambda}{\alpha - 1} \right) \frac{\log2}{2}.
\end{align}

Next, we examine the asymptotic limit of $U_2'(z)$ as $z \to -\infty$.  
From \qref{I-1}, we have
\begin{align}\label{I1-left}
	I_1 (z) = \frac{1}{2} \log \left(\frac{\alpha}{\alpha + 1}\right) + o(1) \quad \text{as } z \to -\infty.
\end{align}
Since $z + ze^{2z} = z + o(1)$, $1 + e^{2z} = 1 + o(1)$, and $\alpha + e^{2z} = \alpha + o(z^{-10})$ as $z \to -\infty$, 
applying Lemma~\ref{rational-lemma}, we get
\begin{align}
	\label{z/D-left} \frac{z}{D(z)} &= \frac{z + ze^{2z}}{\alpha + e^{2z}} = \frac{z}{\alpha} + o(1) \quad \text{as } z \to -\infty, \\
	\label{1/D-left} \frac{1}{D(z)} &= \frac{1 + e^{2z}}{\alpha + e^{2z}} = \frac{1}{\alpha} + o(1) \quad \text{as } z \to -\infty.
\end{align}
On the other hand,
\begin{align}
	\frac{I_2(z)}{D(z)} &= \frac{(\log(1 + e^{2z}) - \log2) (1 + e^{2z})}{2(\alpha + e^{2z})}.
\end{align}
For $z < 0$ with $|z|$ sufficiently large, applying \qref{Taylor-series-1} with $y = e^{2z}$, we obtain
\begin{align}\label{log1-expan-left}
	\log(1 + e^{2z})= e^{2z} - \frac{1}{2} (e^{2z})^2 + \frac{1}{3} (e^{2z})^3 - \frac{1}{4} (e^{2z})^4 + \cdots  
\end{align}
Hence,
\begin{align}
	&(\log(1 + e^{2z}) - \log2) (1 + e^{2z})  \nonumber  \\
	&= \log(1 + e^{2z}) - \log2 + e^{2z}\log(1 + e^{2z}) - e^{2z}\log2  \nonumber  \\
	&= \left( e^{2z} - \frac{1}{2} (e^{2z})^2 + \cdots  \right) - \log2 + e^{2z}\left( e^{2z} - \frac{1}{2} (e^{2z})^2 + \cdots  \right) - e^{2z}\log2, 
\end{align}
which implies that
\begin{align}
	(\log(1 + e^{2z}) - \log2) (1 + e^{-2z}) 
	= - \log2  + o(1) \quad \text{as } z \to -\infty.
\end{align}
Applying Lemma~\ref{rational-lemma} again, we get
\begin{align}\label{I2/D-left}
	\frac{I_2(z)}{D(z)} = \frac{-\log2}{2\alpha}  + o(1) \quad \text{as } z \to -\infty.
\end{align}
Substituting \qref{I1-left}, \qref{z/D-left}, \qref{1/D-left} and \qref{I2/D-left} into \qref{U'_2-b}, and simplifying the resulting expression, we obtain
\begin{align}
	U_2'(z) =& \; \frac{C_{0,0} \beta - h_0}{\alpha}z 
	+ \left( -\frac{\kappa^2 C_{0,0} + \kappa\lambda}{\alpha - 1} - \frac{\kappa C_{1,1}}{\alpha}\right) \frac{1}{2} \log\left(\frac{\alpha}{\alpha + 1}\right) - \frac{\kappa C_{1,0}}{2\alpha}  \nonumber \\ 
	& + \frac{C_{2,1}}{\alpha} - \left( C_{0,0}(\gamma - \beta) + h_0 - q_0 + \frac{\kappa^2 C_{0,0} + \kappa\lambda}{\alpha - 1} \right) \frac{\log2}{2\alpha}  \nonumber  \\
    & + o(1) \quad \text{as } z \to \infty.
\end{align}
Applying Lemma~\ref{integral-lemma}, we get
\begin{align}
	U_2(z) =& \; \frac{C_{0,0} \beta - h_0}{2\alpha} z^2 
	+ \biggl[ \left( -\frac{\kappa^2 C_{0,0} + \kappa\lambda}{\alpha - 1} - \frac{\kappa C_{1,1}}{\alpha}\right) \frac{1}{2} \log\left(\frac{\alpha}{\alpha + 1}\right) - \frac{\kappa C_{1,0}}{2\alpha}  \nonumber \\ 
	& + \frac{C_{2,1}}{\alpha} - \left( C_{0,0}(\gamma - \beta) + h_0 - q_0 + \frac{\kappa^2 C_{0,0} + \kappa\lambda}{\alpha - 1} \right) \frac{\log2}{2\alpha} \biggr]z  \nonumber \\ 
	& + o(z) \quad \text{as } z \to -\infty.
\end{align}
Then, using the matching condition
\begin{align}
	\label{match-U2-L} U_2(z) + o(1) &= u_{L,2}(0) + u'_{L,1}(0)z + \frac{1}{2} u''_{L,0}(0)z^2 \quad \text{as } z \to \infty,
\end{align}
we obtain 
\begin{align}\label{u'-L1}
	u'_{L,1}(0) =& \left( -\frac{\kappa^2 C_{0,0} + \kappa\lambda}{\alpha - 1} - \frac{\kappa C_{1,1}}{\alpha}\right) \frac{1}{2} \log\left(\frac{\alpha}{\alpha + 1}\right) - \frac{\kappa C_{1,0}}{2\alpha} + \frac{C_{2,1}}{\alpha}  \nonumber \\ 
	& - \left( C_{0,0}(\gamma - \beta) + h_0 - q_0 + \frac{\kappa^2 C_{0,0} + \kappa\lambda}{\alpha - 1} \right) \frac{\log2}{2\alpha}.
\end{align}

From \qref{u'-R1} and \qref{u'-L1}, we get
\begin{align}\label{u'R1(0)-au'R1(0)-a}
	&u'_{R,1}(0) - \alpha u'_{L,1}(0)   \nonumber  \\
    &= \left( \frac{\kappa^2 C_{0,0} + \kappa\lambda}{\alpha - 1} + \frac{\kappa C_{1,1}}{\alpha}\right) \frac{1}{2} [\alpha \log\alpha + (1 - \alpha) \log(\alpha + 1)] + \kappa C_{1,0}.
\end{align}
Using \qref{u'R0(0)-au'R0(0)} and \qref{uL1(0)-uR1(0)}, and recalling that $C_{1,1} = \alpha u'_{L,0}(0)$, 
we obtain
\begin{align}\label{u'R1(0)-au'R1(0)-b}
	\frac{1}{2} \left( \frac{\kappa^2 C_{0,0} + \kappa\lambda}{\alpha - 1} + \frac{\kappa C_{1,1}}{\alpha}\right) 
	&= \frac{\kappa}{\alpha - 1} \left( \frac{\kappa C_{0,0} + \lambda}{2} + C_{1,1}\frac{\alpha - 1}{2\alpha} \right)   \nonumber \\ 
	&= \frac{\kappa}{\alpha - 1} \frac{u_{L,1}(0) - u_{R,1}(0)}{\log\alpha}.
\end{align}
Moreover, from \qref{uR1(0)} and \qref{uL1(0)-uR1(0)}, we get
\begin{align}\label{C1,0}
	C_{1,0} = u_{R,1}(0) + [u_{L,1}(0) - u_{R,1}(0)] \frac{\log(\alpha + 1)}{\log\alpha}.
\end{align}
Substituting \qref{u'R1(0)-au'R1(0)-b} and \qref{C1,0} into \qref{u'R1(0)-au'R1(0)-a}, and simplifying the resulting expression, we obtain
\begin{align}\label{u'R1(0)-au'R1(0)}
	u'_{R,1}(0) - \alpha u'_{L,1}(0) = \frac{\kappa}{\alpha - 1} [u_{R,1}(0) - \alpha u_{L,1}(0)].
\end{align}
From \qref{outer-u-Lj}, \qref{outer-u-Rj}, \qref{uL1(0)-uR1(0)-b} and \qref{u'R1(0)-au'R1(0)}, we see that: 
When $\alpha \neq 1$, $u_{R,1}(x)$ and $u_{L,1}(x)$ are determined by the following two-sided problem
\begin{align}
	\label{u1RL-case1-eqn1} -u_{R,1}'' + \gamma u_{R,1} &= 0, \quad \text{in } \Omega_R, \\
	-\alpha u_{L,1}'' + \beta u_{L,1} &= 0, \quad \text{in } \Omega_L, \\
        u_{L,1}(0) - u_{R,1}(0) &= \frac{\log\alpha}{2}[u_{R,0}' (0) - u_{L,0}' (0)] \\
	u'_{R,1}(0) - \alpha u'_{L,1}(0) &= \frac{\kappa}{\alpha - 1} [u_{R,1}(0) - \alpha u_{L,1}(0)], \\
	\label{u1RL-case1-eqn2} u_{L,1}'(-1) &= 0 = u_{R,1}'(1).
\end{align}
This shows that, when $\alpha \neq 1$, the diffuse domain problem is first-order accurate, 
as the solution pair $(u_{R,1}, u_{L,1})$ of the problem \qref{u1RL-case1-eqn1}--\qref{u1RL-case1-eqn2} is generally not equal to $(0,0)$.

\

\noindent $\bullet$ \textbf{Case 2: $\alpha = 1$.} 

In this case, $D(z) \equiv 1$ and $I_1(z) = I_2(z)$. Hence, Equation~\qref{D-U'2} becomes
    \begin{align}
    \label{U'2-case2-a}
U_2'(z) & = C_{0,0}[\beta z + (\gamma - \beta)I_2(z)] - h_0 z + (h_0 - q_0)I_2(z) 
    \nonumber 
    \\
& \quad + (\kappa^2 C_{0,0} +\kappa\lambda) \int_0^z I_2(s) \phi'(s)\;ds + \kappa C_{1,1} \int_0^z \phi'(s)s \;ds   
    \nonumber 
    \\
& \quad + \kappa C_{1,0} \phi(z) - \frac{\kappa C_{1,0}}{2} + C_{2,1}. 
    \end{align}
Using integration by parts, we get
    \begin{align}
    \label{I2-phi'}
\int_0^z I_2(s)\phi'(s) \;ds & = I_2(s)\phi(s) \bigg\rvert_0^z - \int_0^z \phi(s) I'_2(s) \;ds 
    \nonumber 
    \\
& = I_2(z)\phi(z) - \int_0^z (\phi(s))^2 \;ds 
    \nonumber 
    \\
& = I_2(z)\phi(z) - \int_0^z \frac{e^{2s}}{(e^s + e^{-s})^2}\; ds  
    \nonumber
    \\
& = I_2(z)\phi(z) - \frac{1}{2} \log(e^{2z} + 1) - \frac{1}{2(e^{2z} + 1)} + \frac{\log2}{2} + \frac{1}{4}.
    \end{align}
Substituting \qref{s-phi's} and \qref{I2-phi'} into \qref{U'2-case2-a}, and simplifying the resulting expression, we get
    \begin{align}
    \label{U'2-case2}
U_2'(z) & = (C_{0,0}\beta - h_0) z + [C_{0,0}(\gamma - \beta) + h_0 - q_0 - \kappa C_{1,1}]I_2(z) 
    \nonumber 
    \\
& \quad + (\kappa^2 C_{0,0} +\kappa\lambda)\left( I_2(z)\phi(z) - \frac{1}{2} \log(e^{2z} + 1) - \frac{1}{2(e^{2z} + 1)} + \frac{\log2}{2} + \frac{1}{4} \right) 
    \nonumber 
    \\
& \quad  + \kappa C_{1,1} \phi(z)z + \kappa C_{1,0} \phi(z) - \frac{\kappa C_{1,0}}{2} + C_{2,1}. 
    \end{align}

Again, we examine the asymptotic limit of $U_2'(z)$ as $z \to \infty$ first.  
From \qref{log1-expan-right}, we have
\begin{align}\label{log1-right}
	\log(e^{2z} + 1) = z + o(1) \quad \text{as } z \to \infty.
\end{align}
Using \qref{I-1} with $\alpha = 1$, we get 
\begin{align}\label{I2-right}
	I_2(z) = z - \frac{\log2}{2} + o(1) \quad \text{as } z \to \infty.
\end{align}
Combining the equation above with $1 + e^{-2z} = 1 + o(z^{-10})$ as $z \to \infty$, 
then applying Lemma~\ref{rational-lemma}, we get
\begin{align}\label{I2-phi-right}
	I_2(z)\phi(z) = \frac{I_2(z)}{1 + e^{-2z}} = z - \frac{\log2}{2} + o(1) \quad \text{as } z \to \infty.
\end{align}
Similarly, we obtain
\begin{align}
	\label{phi-right} \phi(z) &= \frac{1}{1 + e^{-2z}} = 1 + o(1) \quad \text{as } z \to \infty, \\
	\label{z-phi-right} \phi(z)z &= \frac{z}{1 + e^{-2z}} = z + o(1) \quad \text{as } z \to \infty, \\
	\label{1/exp-right} \frac{1}{e^{2z} + 1} &= \frac{e^{-2z}}{1 + e^{-2z}} = 0 + o(1) \quad \text{as } z \to \infty.
\end{align}
Substituting \qref{log1-right}, \qref{I2-right}, \qref{I2-phi-right}, \qref{phi-right}, \qref{z-phi-right} and \qref{1/exp-right} into \qref{U'2-case2}, and simplifying the resulting expression, we get
    \begin{align}
U'_2(z) & = (C_{0,0}\gamma - q_0)z - [C_{0,0}(\gamma - \beta) + h_0 - q_0 - \kappa C_{1,1}]\frac{\log2}{2} + \kappa C_{1,0} 
    \nonumber 
    \\
& \quad  + \frac{1}{4}(\kappa^2 C_{0,0} + \kappa\lambda) - \frac{\kappa C_{0,0}}{2} + C_{2,1} + o(1) \quad \text{as } z \to \infty.
    \end{align}
Therefore, applying Lemma~\ref{integral-lemma}, we obtain
    \begin{align}
U_2(z) & = \frac{1}{2}(C_{0,0}\gamma - q_0)z^2 + \biggl[ -[C_{0,0}(\gamma - \beta) + h_0 - q_0 - \kappa C_{1,1}]\frac{\log2}{2} + \kappa C_{1,0} 
    \nonumber 
    \\
& \quad  + \frac{1}{4}(\kappa^2 C_{0,0} + \kappa\lambda) - \frac{\kappa C_{0,0}}{2} + C_{2,1} \biggr] z + o(z) \quad \text{as } z \to \infty.
    \end{align}
Then, using the matching condition
    \begin{align}
U_2(z) + o(1) & = u_{R,2}(0) + u'_{R,1}(0)z + \frac{1}{2} u''_{R,0}(0)z^2 \quad \text{as } z \to \infty,
    \end{align}
we get 
    \begin{align}
    \label{u'-R1-case2}
u'_{R,1}(0)  =& -[C_{0,0}(\gamma - \beta) + h_0 - q_0 - \kappa C_{1,1}]\frac{\log2}{2} + \kappa C_{1,0}  \nonumber  \\
    & + \frac{1}{4}(\kappa^2 C_{0,0} + \kappa\lambda) - \frac{\kappa C_{0,0}}{2} + C_{2,1}.
    \end{align}

Next, we examine the asymptotic limit of $U_2'(z)$ as $z \to -\infty$.  
From \qref{log1-expan-left}, we have
\begin{align}\label{log1-left}
	\log(e^{2z} + 1) = 0 + o(1) \quad \text{as } z \to -\infty.
\end{align}
Using \qref{I-1} with $\alpha = 1$, we get 
\begin{align}\label{I2-left}
	I_2(z) = - \frac{\log2}{2} + o(1) \quad \text{as } z \to -\infty.
\end{align}
From \qref{I-2}, we have
\begin{align}
	I_2(z)\phi(z) = \frac{[\log(e^{2z} + 1) - \log2]e^{2z}}{2(e^{2z} + 1)}.
\end{align}
Using \qref{log1-expan-left} for $z < 0$ with $|z|$ sufficiently large, we get
\begin{align*}
	[\log(e^{2z} + 1) - \log2]e^{2z} = -(\log2)e^{2z} + (e^{2z})^2 - \frac{1}{2}(e^{2z})^3 + \frac{1}{3}(e^{2z})^4 - \cdots ,
\end{align*}
which implies that
\begin{align}
	[\log(e^{2z} + 1) - \log2]e^{2z} = 0 + o(1) \quad \text{as } z \to -\infty.
\end{align}
Combining the equation above with $1 + e^{2z} = 0 + o(z^{-10})$ as $z \to -\infty$, 
then applying Lemma~\ref{rational-lemma}, we get
\begin{align}\label{I2-phi-left}
	I_2(z)\phi(z) = 0 + o(1) \quad \text{as } z \to -\infty.
\end{align}
Similarly, we obtain
\begin{align}
	\label{phi-left} \phi(z) = \frac{e^{2z}}{e^{2z} + 1} &= 0 + o(1) \quad \text{as } z \to -\infty, \\
	\label{z-phi-left} \phi(z)z = \frac{z e^{2z}}{e^{2z} + 1}  &= 0 + o(1) \quad \text{as } z \to -\infty, \\
	\label{1/exp-left} \frac{1}{e^{2z} + 1} &= 1 + o(1) \quad \text{as } z \to -\infty.
\end{align}
Substituting \qref{log1-left}, \qref{I2-left}, \qref{I2-phi-left}, \qref{phi-left}, \qref{z-phi-left} and \qref{1/exp-left} into \qref{U'2-case2}, and simplifying the resulting expression, we get
\begin{align}
	U'_2(z) =& \; (C_{0,0}\beta - h_0)z - [C_{0,0}(\gamma - \beta) + h_0 - q_0 - \kappa C_{1,1}]\frac{\log2}{2} \nonumber \\
	& + \left( \frac{\log2}{2} - \frac{1}{4} \right) (\kappa^2 C_{0,0} + \kappa\lambda) - \frac{\kappa C_{0,0}}{2} + C_{2,1}
	+ o(1) \quad \text{as } z \to -\infty.
\end{align}
Therefore, applying Lemma~\ref{integral-lemma}, we obtain
\begin{align}
	U_2(z) =& \; \frac{1}{2}(C_{0,0}\beta - h_0)z^2 + \biggl[ -[C_{0,0}(\gamma - \beta) + h_0 - q_0 - \kappa C_{1,1}]\frac{\log2}{2} \nonumber \\
	& + \left( \frac{\log2}{2} - \frac{1}{4} \right) (\kappa^2 C_{0,0} + \kappa\lambda) - \frac{\kappa C_{0,0}}{2} + C_{2,1} \biggr]z 
	  \nonumber  \\
    &+ o(z) \quad \text{as } z \to -\infty.
\end{align}
Then, using the matching condition
\begin{align}
	U_2(z) + o(1) &= u_{L,2}(0) + u'_{L,1}(0)z + \frac{1}{2} u''_{L,0}(0)z^2 \quad \text{as } z \to -\infty,
\end{align}
we get 
\begin{align}\label{u'-L1-case2}
	u'_{L,1}(0) =&  -[C_{0,0}(\gamma - \beta) + h_0 - q_0 - \kappa C_{1,1}]\frac{\log2}{2}   \nonumber  \\
    & + \left( \frac{\log2}{2} - \frac{1}{4} \right) (\kappa^2 C_{0,0} + \kappa\lambda) - \frac{\kappa C_{0,0}}{2} + C_{2,1}.
\end{align}
Combining \qref{u'-R1-case2} and \qref{u'-L1-case2}, we obtain
\begin{align}\label{u'L1(0)-u'R1(0)}
	u'_{R,1}(0) - u'_{L,1}(0) = \kappa C_{1,0} + \left( \frac{1}{2} - \frac{\log2}{2} \right) (\kappa^2 C_{0,0} + \kappa\lambda).
\end{align}
Using \qref{u'R0(0)-au'R0(0)} and \qref{uR1(0)} with $\alpha = 1$, we have
\begin{align}
	\label{u'R0(0)-u'R0(0)}  \kappa u_{R,0} (0) + \lambda &= u_{R,0}' (0) - u_{L,0}' (0), \\
	\label{C1,0-case2} C_{1,0} &= u_{R,1}(0) + (\kappa C_{0,0} + \lambda)\frac{\log2}{2}.
\end{align}
Since $C_{0,0} = u_{R,0}(0)$, by substituting \qref{u'R0(0)-u'R0(0)} and \qref{C1,0-case2} into \qref{u'L1(0)-u'R1(0)}, we get
\begin{align}\label{u'L1(0)-u'R1(0)-b}
	u'_{R,1}(0) - u'_{L,1}(0) = \kappa u_{R,1}(0) + \frac{\kappa}{2}[u_{R,0}' (0) - u_{L,0}' (0)].
\end{align}
On the other hand, using \qref{uL1(0)-uR1(0)-b} with $\alpha = 1$, we have
\begin{align}
	u_{R,1}(0) - u_{L,1}(0) = 0.
\end{align}
Combining the equation above with \qref{outer-u-Lj}, \qref{outer-u-Rj} and \qref{u'L1(0)-u'R1(0)-b}, we see that: 
When $\alpha = 1$, $u_{R,1}(x)$ and $u_{L,1}(x)$ are determined by the following two-sided problem
\begin{align}
	\label{u1RL-case2-eqn1} -u_{R,1}'' + \gamma u_{R,1} &= 0, \quad \text{in } \Omega_R, \\
	-u_{L,1}'' + \beta u_{L,1} &= 0, \quad \text{in } \Omega_L, \\
	u_{R,1}(0) &= u_{L,1}(0), \\
        u'_{R,1}(0) - u'_{L,1}(0) &= \kappa u_{R,1}(0) + \frac{\kappa}{2}[u_{R,0}' (0) - u_{L,0}' (0)], \\
	\label{u1RL-case2-eqn2} u_{L,1}'(-1) &= 0 = u_{R,1}'(1).
\end{align}
This shows that, when $\alpha = 1$, the diffuse domain problem is first-order accurate, 
as the solution pair $(u_{R,1}, u_{L,1})$ of the problem \qref{u1RL-case2-eqn1}--\qref{u1RL-case2-eqn2} is generally not equal to $(0,0)$.

\section{Numerical Simulations}
To assess the order of convergence of the diffuse domain method, we generate reference solutions for the corresponding sharp interface problems and perform error analyses.

\subsection{Numerical Simulations in 1D}
For a 1D problem, we define the domains as 
\begin{align*}
    \Omega:=(-1,1), \quad \Omega_1=\Omega_R:=(0,1), \quad \Omega_2=\Omega_L:=(-1,0), 
\end{align*}
and choose the function
\begin{align}\label{true-sln-1D}
    u_0(x) = \begin{cases}	 
    u_R(x) = (4x^2-8x+6)\cos{(4\pi x)} ,&  \text{ if }  x\in\Omega_R,
    \\
    u_L(x) = 8(x+1)^2-2 ,&  \text{ if } x\in \Omega_L, 
    \end{cases}
\end{align}
which is a solution to the two-sided problem \qref{1D-bvp1}--\qref{1D-bvp5}, where
    \begin{align}
\label{data1D-1} \alpha & = 0.5, \quad \beta = \gamma = 1, \quad \kappa = 1.6, \quad \lambda = -25.6, 
    \\
q(x) & = (16\pi^2 + 1) (4x^2 - 8x + 6)\cos(4 \pi x) + 64 \pi (x - 1) \sin(4 \pi x), 	 
    \\
\label{data1D-2} h(x) & = 8(x + 1)^2 - 10.
    \end{align}
The functions $u_L(x)$ and $u_R(x)$ are continuous in $\Omega_L$ and $\Omega_R$, respectively, 
and match at $x=0$, ensuring that $u_0(x)$ is globally continuous across $\Omega$. 
Additionally, we have
\[
    u_R' (0) - \alpha u_L' (0) = -16, 
\]
indicating a non-zero jump in the flux across the interface $x = 0$. 

To numerically solve the diffuse domain approximation problem \qref{1D-diff-bvp1}--\qref{1D-diff-bvp2},  we implement a second-order cell-centered finite difference approximation with a multigrid solver~\cite{feng2018}.  The plots of the true solution $u_0$ and the approximation solution $u_\varepsilon$, for various $\varepsilon$ values, are presented in Figure~\ref{fig:Sol_WholeDomain}. We observe that, as $\varepsilon \to 0$, $u_\varepsilon$ converges to $u_0$.  To analyze the $L^2$ convergence rate, we fix the grid size of the finite difference method to $N=2.0\times 10^5$, run the solver for progressively smaller values of $\varepsilon$, and compute the discrete $L^2$ error for each case. The log-log plot of the error is shown in Figure~\ref{fig:ErrorPlot}. We compute a numerical convergence rate of $O(\varepsilon^{1.1784})$ as measured in a discrete $L^2$ norm. This is consistent with the first-order convergence rate suggested by the asymptotic analysis.


\begin{figure}[htb!]
	\centering
	\includegraphics[width=0.75\linewidth]{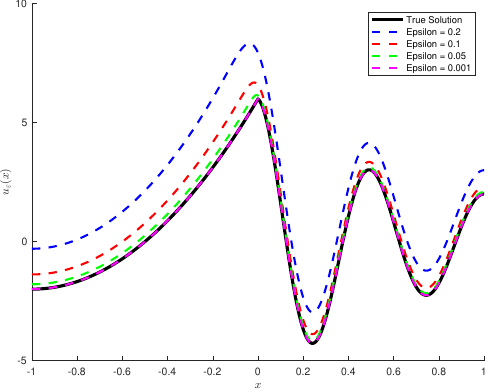}
	\caption{Plots of the true solution $u_0$ specified by \qref{true-sln-1D} and the diffuse domain approximation solution $u_\varepsilon$ over $\Omega = (-1,1)$, for various values of $\varepsilon$, with the parameters given in  \qref{data1D-1}--\qref{data1D-2}.}
	\label{fig:Sol_WholeDomain}
\end{figure}

\begin{figure}[htb!]
	\centering
	\includegraphics[width=0.8\textwidth]{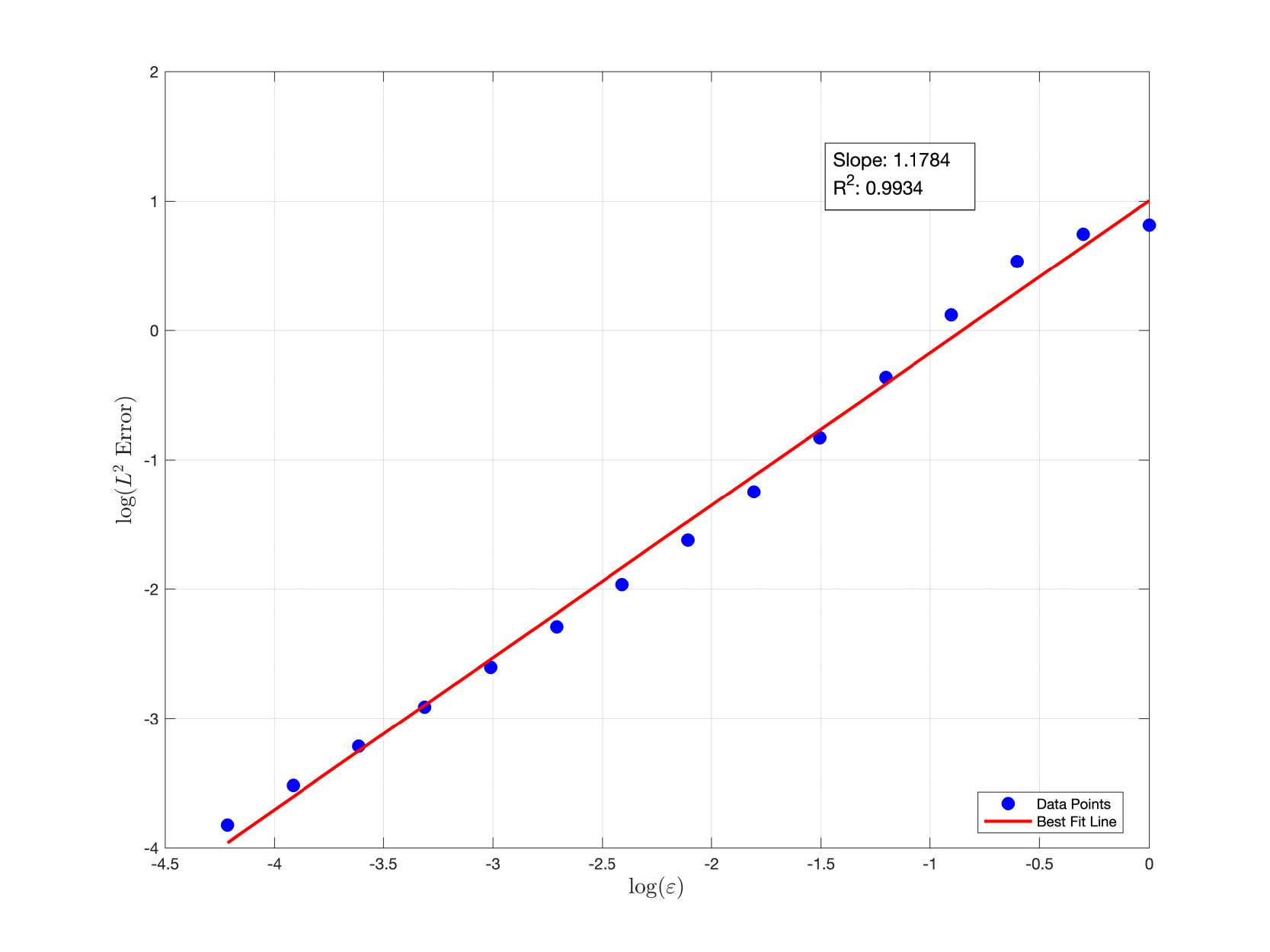}
	\caption{Log-log plot of the $L^2$ Error versus $\varepsilon$ for the 1D problem described by  \qref{data1D-1}--\qref{data1D-2}.}
	\label{fig:ErrorPlot}
\end{figure}

Next, we investigate a coupled refinement strategy where the grid size $h$ and the interface parameter $\varepsilon$ are related via $h=r\varepsilon$, $0<r<1$. Choosing $r$ sufficiently small ensures the computational grid resolves the diffuse interface region as $\varepsilon$ decreases. For this study, we consider the 1D two-sided problem defined by the PDE system \qref{1D-bvp1}--\qref{1D-bvp5}. We construct a test case with a known exact solution $u_0(x)$, given in \qref{true-sln-1D}, over the domain $\Omega = (-1, 1)$ with the interface at $x=0$. This solution corresponds to the specific parameter choices and function definitions detailed in \qref{data1D-1}--\qref{data1D-2}. We present numerical results for two specific refinement paths: $r=1/10$ (corresponding to $\varepsilon=10h$) and $r=1/20$ (corresponding to $\varepsilon=20h$). Figure~\ref{fig:adaptive_10} displays the $L^2$ and $L^\infty$ error norms $\|u_\varepsilon - u_0\|$ versus $\varepsilon$ for the $\varepsilon=10h$ case, while Figure~\ref{fig:adaptive_20} shows the corresponding results for $\varepsilon=20h$. In both scenarios, the observed convergence rate is clearly first-order ($O(\varepsilon)$), consistent with theoretical expectations and demonstrating the stability and accuracy of the diffuse domain method with these refinement paths when compared against the exact solution.

\begin{figure}[htbp]
    \centering
    \begin{subfigure}[b]{0.49\textwidth}
        \centering
        \includegraphics[width=\textwidth]{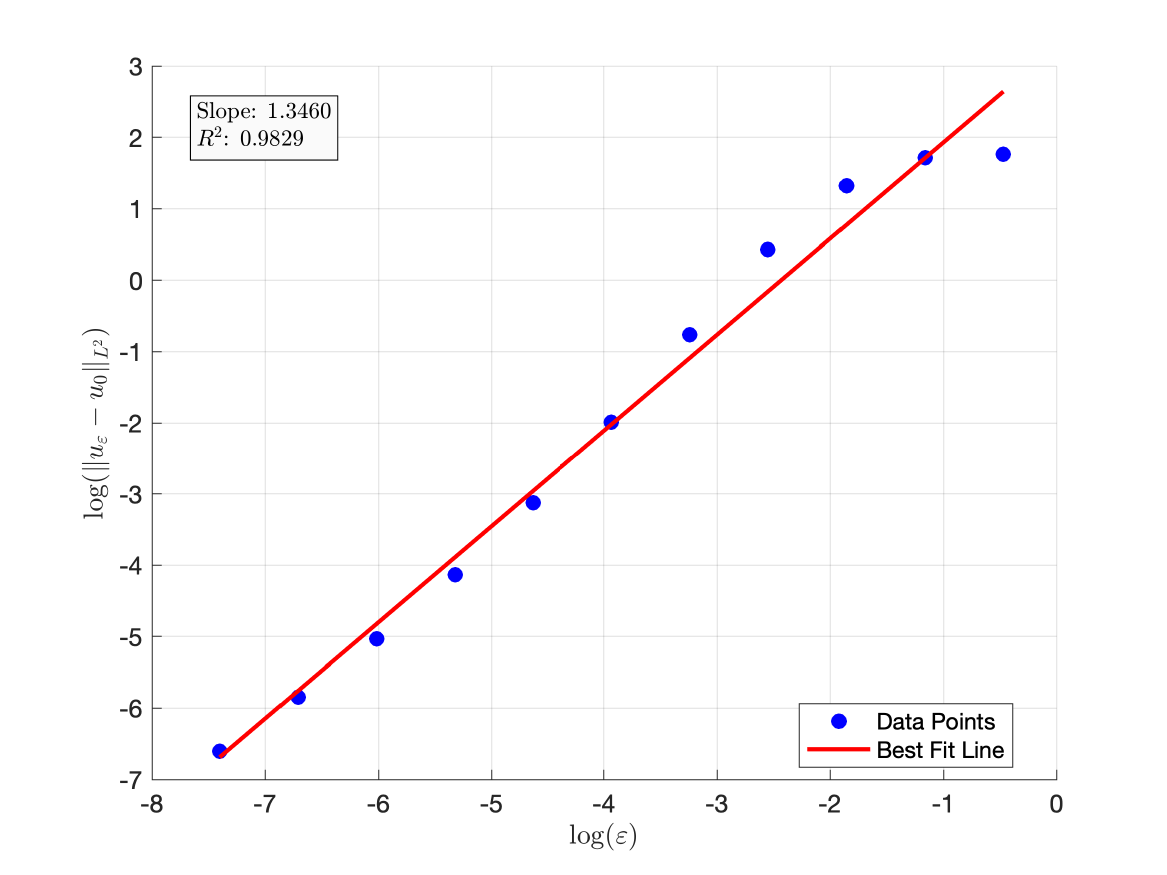}
        \caption{$L^2$ Error for $\varepsilon=10h$.} 
        \label{fig:subfig_adaptive_10_L2}
    \end{subfigure}
    \hfill 
    \begin{subfigure}[b]{0.49\textwidth}
        \centering
        \includegraphics[width=\textwidth]{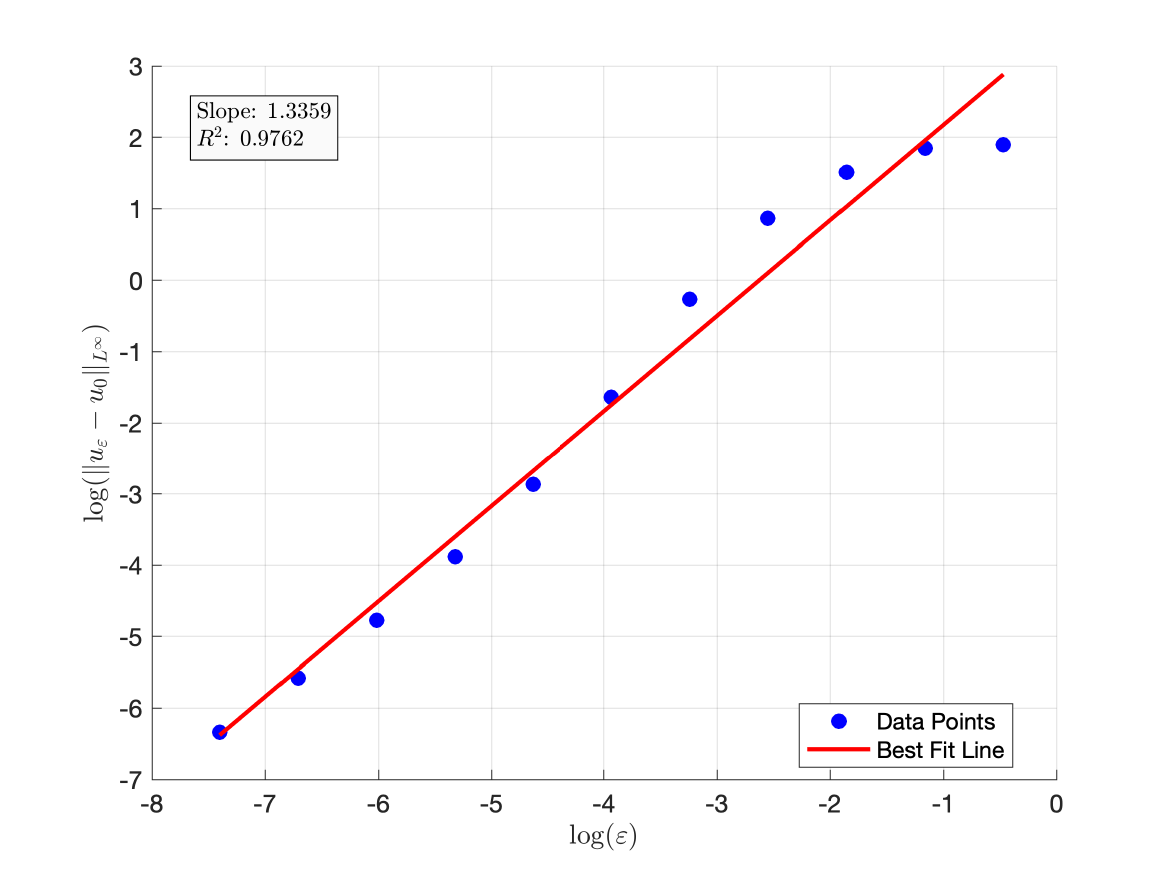}
        \caption{$L^\infty$ Error for $\varepsilon=10h$.} 
        \label{fig:subfig_adaptive_10_Linf}
    \end{subfigure}
    \caption{Log-log plot of errors versus $\varepsilon$ using the $\varepsilon=10h$ refinement path. First-order convergence is observed.} 
    \label{fig:adaptive_10}
\end{figure}

\begin{figure}[htbp]
    \centering
    \begin{subfigure}[b]{0.49\textwidth}
        \centering
        \includegraphics[width=\textwidth]{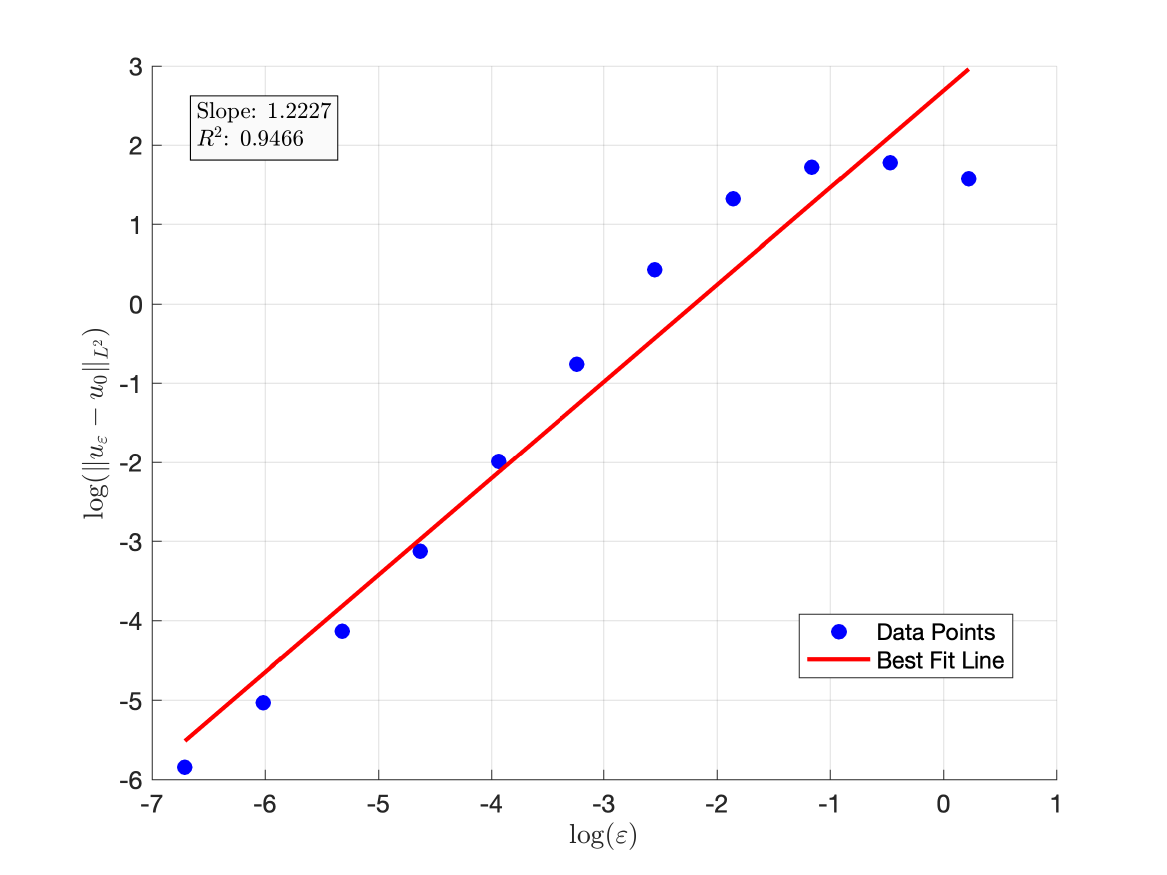} 
        \caption{$L^2$ Error for $\varepsilon=20h$.} 
        \label{fig:subfig_adaptive_20_L2}
    \end{subfigure}
    \hfill 
    \begin{subfigure}[b]{0.49\textwidth}
        \centering
        \includegraphics[width=\textwidth]{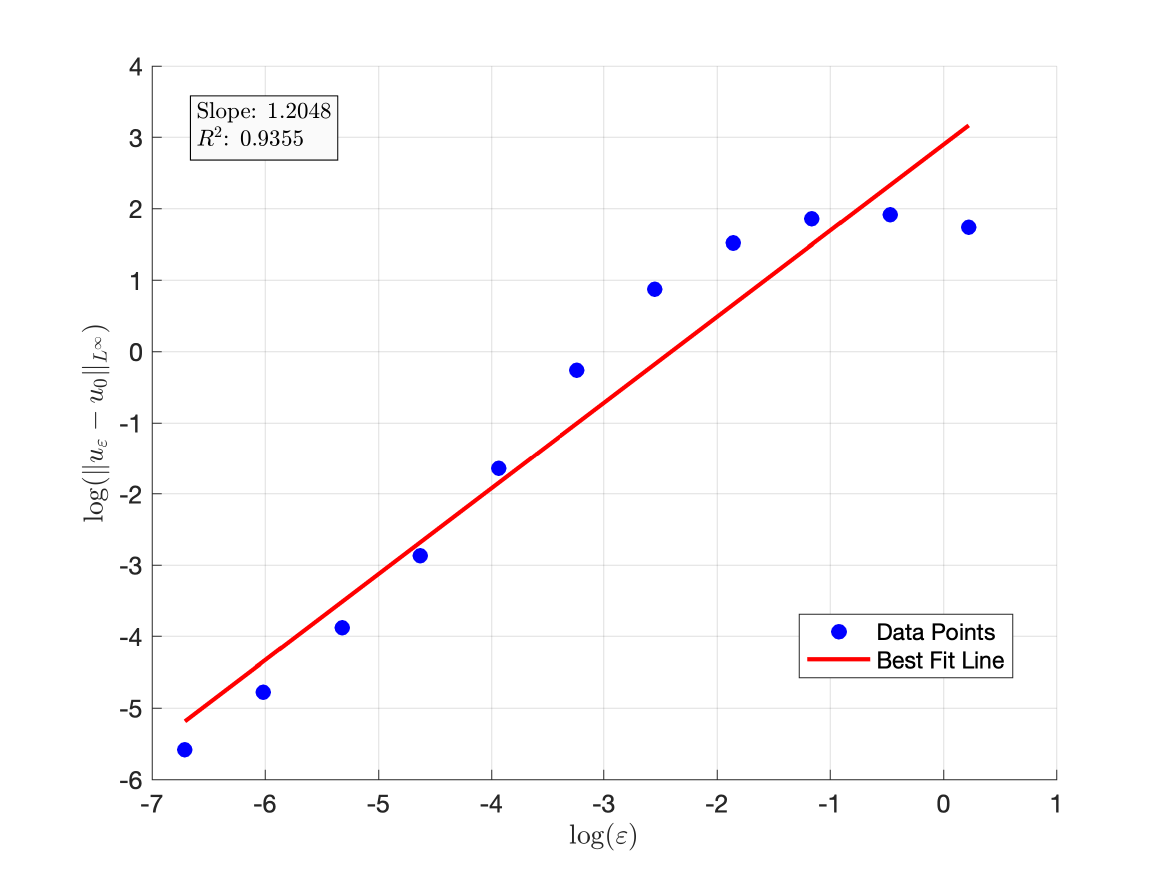} 
        \caption{$L^\infty$ Error for $\varepsilon=20h$.} 
        \label{fig:subfig_adaptive_20_Linf}
    \end{subfigure}
    \caption{Log-log plot of errors versus $\varepsilon$ using the $\varepsilon=20h$ refinement path. First-order convergence is maintained.} 
    \label{fig:adaptive_20}
\end{figure}


\subsection{Numerical Simulations in 2D}
To analyze the algorithm for a 2D problem, we define the domains as
    \begin{align*}
\Omega_1 &:=\left\{(x_1, x_2) : x_1^2 + 4x_2^2< 4\right\}, \\
\Omega &:= (-3,3) \times (-3,3), \quad \Omega_2 := \Omega \setminus \overline{\Omega_1},
    \end{align*}
and choose the function
\begin{align}\label{true-sln-2D}
	u_0(x_1, x_2) = \begin{cases}	 
		u_1(x_1, x_2)=-x_1^2 - 4x_2^2 + 6 ,&  \text{ if }  (x_1, x_2) \in\Omega_1,
		\\
		u_2(x_1, x_2) = 2 ,&  \text{ if } (x_1, x_2) \in \Omega_2, 
	\end{cases}
\end{align}
which solves the two-sided problem \qref{bvp1}--\qref{bvp5}, with the parameters
    \begin{align}
\label{data2D-1} \alpha & = 0.5, \quad \beta =  \gamma =  \kappa = 1,
    \\
q(x_1, x_2) & = -x_1^2 - 4x_2^2 + 16,
    \quad
h(x_1, x_2)  =2,
	\\
\label{data2D-2} g(x_1, x_2) &= x_1^2 + 4x_2^2 - 6 + 2\sqrt{x_1^2 + 16x_2^2}.
    \end{align}
The plot of $u_0$ is presented in Figure \ref{fig:2DTrueSolutionAngle1}. 
The functions $u_1$ and $u_2$ are smooth in $\Omega_1$ and $\Omega_2$, respectively,  and match on the boundary $\partial\Omega_1$ of $\Omega_1$,  ensuring that $u_0$ is globally continuous across $\Omega$.  Furthermore, we have
\begin{align*}
	 - \nabla (u_1 - \alpha u_2) \cdot \boldsymbol{n}_1  = 2\sqrt{x_1^2 + 16x_2^2},
\end{align*}
which indicates a non-zero jump in the flux accross the interface $\partial\Omega_1$.

To numerically solve the diffuse domain approximation problem \qref{diff-bvp1}--\qref{diff-bvp2} in 2D, first, we need to compute the signed distance function $r(x_1,x_2)$ from $(x_1,x_2) \in \Omega$ to $\partial\Omega_1$ using a numerical approach. We use a fast-marching algorithm~\cite{sethian1996} to solve the following Eikonal equation in 2D:
    \begin{align*}
\left\lvert\nabla r\right\rvert = 1  \text{ in } \Omega,  
    \quad 
r = 0  \text{ on } \partial\Omega_1.
    \end{align*}
Figure \ref{fig:2DDomain} illustrates the plot of the interface $\partial\Omega_1$ alongside the level curves of the signed distance function $r(x_1,x_2)$.

We use a second-order, cell-centered 2D finite difference method to approximate the DDM approximation, and we employ an efficient linear multigrid method to solve the resulting system of linear equations. The details for both the discretization and the multigrid solver are similar to those in~\cite{feng2018}; we suppress the details for the sake of brevity. The plot of $u_\varepsilon$, for $\varepsilon = 0.05$, is presented in Figure~\ref{fig:2DNumericalSolutionAngle1}.

To analyze the $L^2$ convergence rate, we fix the grid size of the finite difference method to $N=1.0\times 10^6$, run the solver for progressively smaller values of $\varepsilon$, and compute the discrete $L^2$ error for each case. The log-log plot of the error is shown in Figure~\ref{fig:2DErrorPlot}. We compute a numerical convergence rate of $O(\varepsilon^{1.1982})$ as measured in a discrete $L^2$ norm. This is consistent with the first-order convergence rate suggested by the asymptotic analysis. 

Finally, to test the robustness of the algorithm for a complex interface $\partial\Omega_1$, we consider the following ``starfish'' problem, 
which is the problem \qref{diff-bvp1}--\qref{diff-bvp2} over the following domains:
    \begin{align}
\label{star-domain} \Omega_1 &:=\left\{(r,\theta):r< 0.9(1.2+0.7\sin(5\theta)), \; 0\leq\theta<2\pi\right\},  
    \\
\Omega &:= (-2,2) \times (-2,2), \quad \Omega_2 := \Omega \setminus \overline{\Omega_1},	
    \end{align}
with the parameters
 \begin{align}
\label{starfish-1} \alpha & = 3, \quad \beta = 2, \quad \gamma = 1, \quad  \kappa = 0.01,
    \\
q(x_1, x_2) & = -x_1^2 + 15,
    \quad
h(x_1, x_2)  = 2.5\sin(x_1) + e^{\cos(x_2)},
	\\
\label{starfish-2} g(x_1, x_2) &= 4.
\end{align}
Here, $(r,\theta)$ are polar coordinates, determined by
\begin{align*}
	x_1 = r\cos\theta, \quad	x_2 = r\sin\theta.
\end{align*}
The plot of the interface $\partial\Omega_1$ and the approximation solution $u_\varepsilon$ to the ``starfish'' problem, computed for $\varepsilon = 0.05$, are presented in Figures~\ref{fig:2DFunDomain} and \ref{fig:2DFunDomainPlot}, respectively.


\begin{figure}[htb!]
    \centering
    \includegraphics[width=0.8\textwidth]{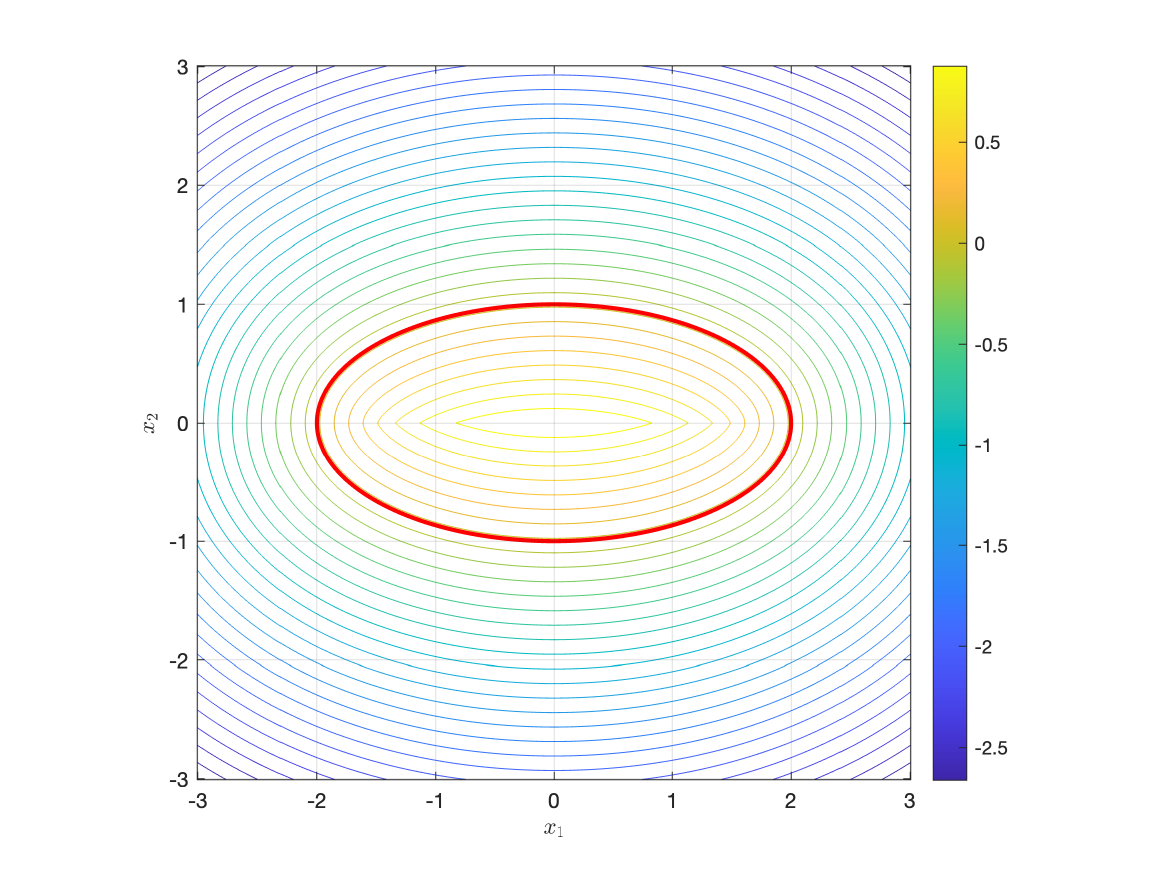}
    \caption{Plot of the interface $\partial\Omega_1 =\left\{(x_1, x_2) : x_1^2 + 4x_2^2= 4\right\}$ and the level curves of the signed distance function $r(x_1,x_2)$ in 2D.}
    \label{fig:2DDomain}
\end{figure}

\begin{figure}[htb!]
    \centering
    \includegraphics[width=0.8\textwidth]{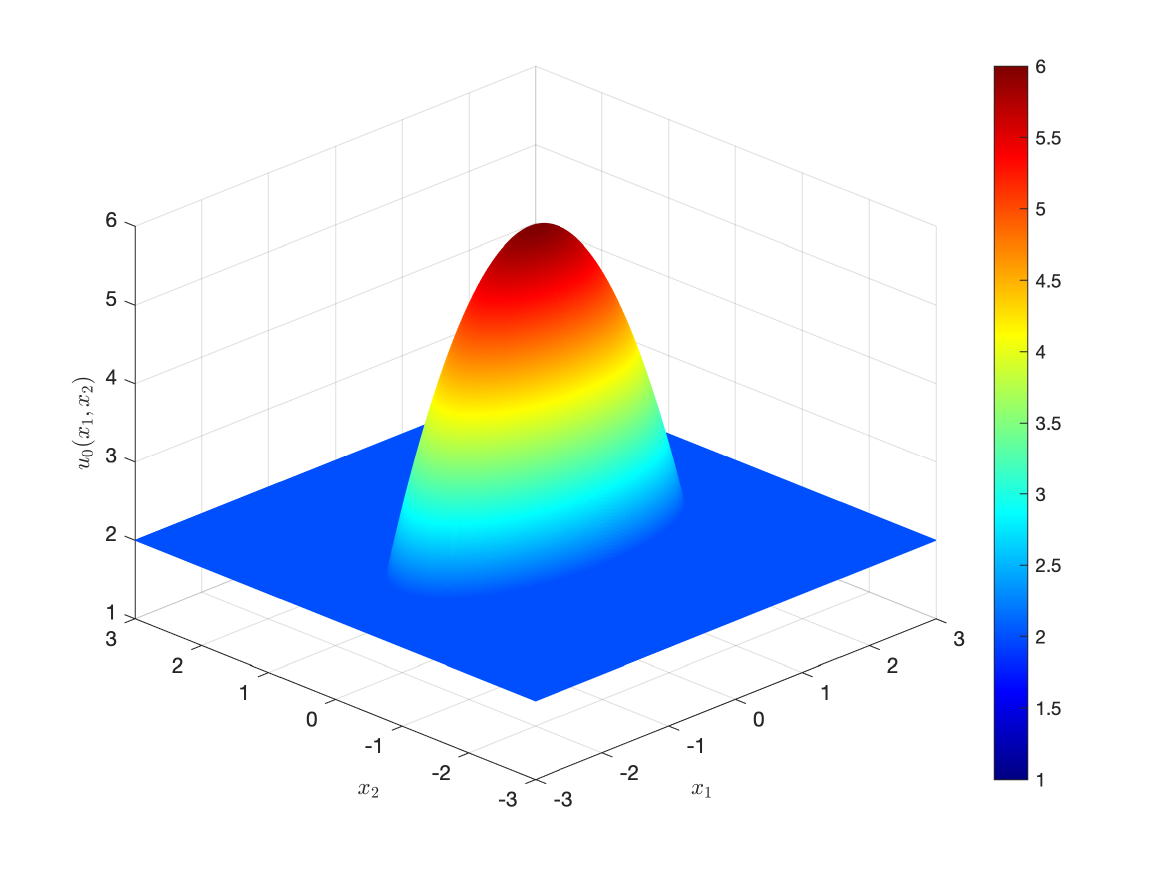}
    \caption{Plot of the true solution $u_0$ specified by \qref{true-sln-2D}.}
    \label{fig:2DTrueSolutionAngle1}
\end{figure}

\begin{figure}[htb!]
    \centering
    \includegraphics[width=0.8\textwidth]{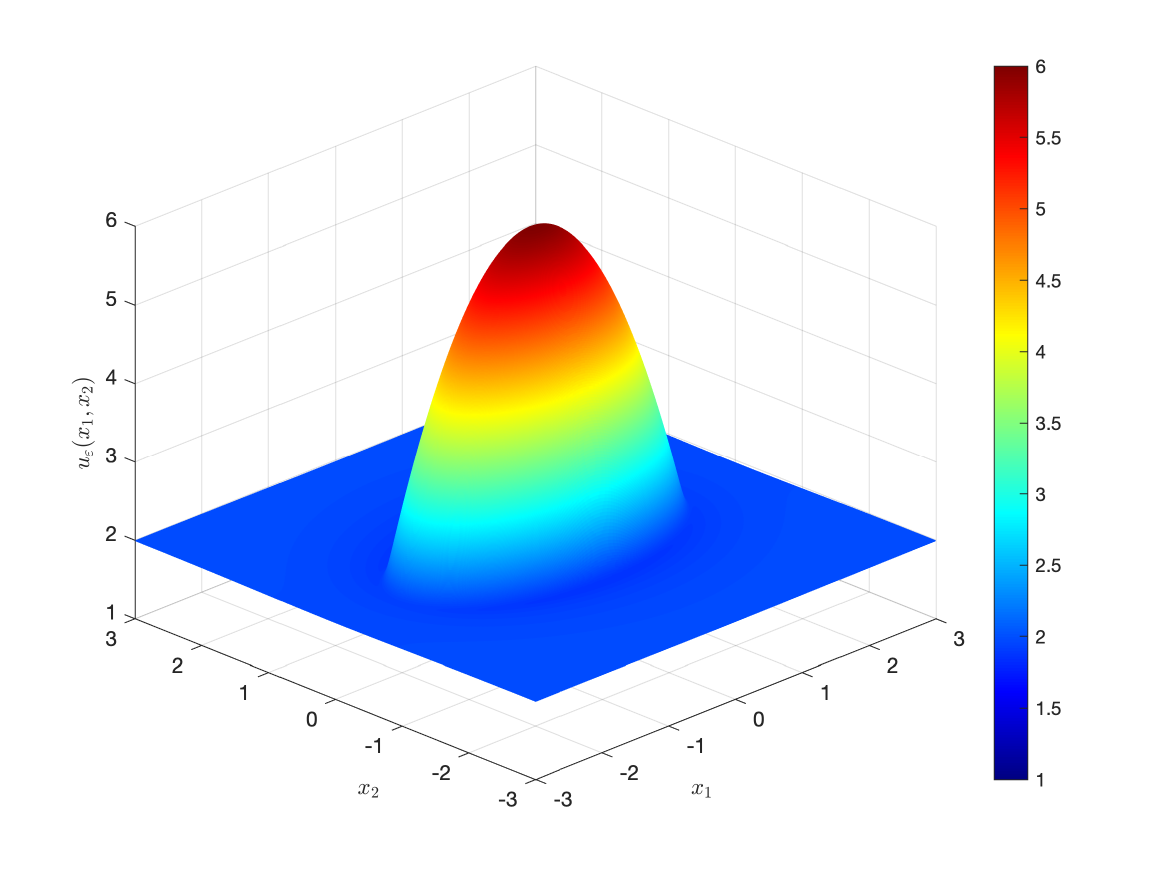}
    \caption{Plot of the diffuse domain approximation solution $u_\varepsilon$, for $\varepsilon=0.05$, with the parameters given in  \qref{data2D-1}--\qref{data2D-2}.}
    \label{fig:2DNumericalSolutionAngle1}
\end{figure}

\begin{figure}[htb!]
    \centering
    \includegraphics[width=0.75\textwidth]{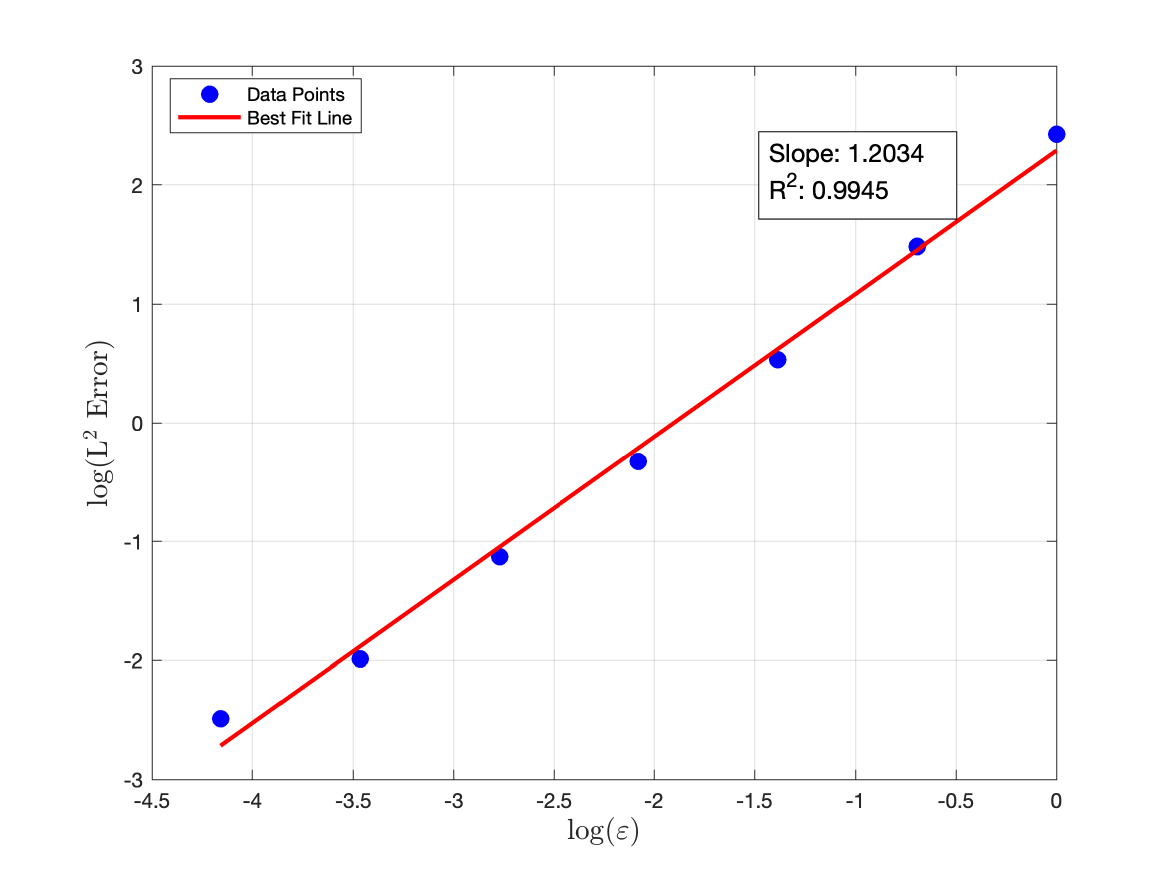}
    \caption{Log-log plot of the $L^2$ Error versus $\varepsilon$ for the 2D problem described by  \qref{data2D-1}--\qref{data2D-2}.}
    \label{fig:2DErrorPlot}
\end{figure}

\begin{figure}[htb!]
    \centering
    \includegraphics[width=0.8\textwidth]{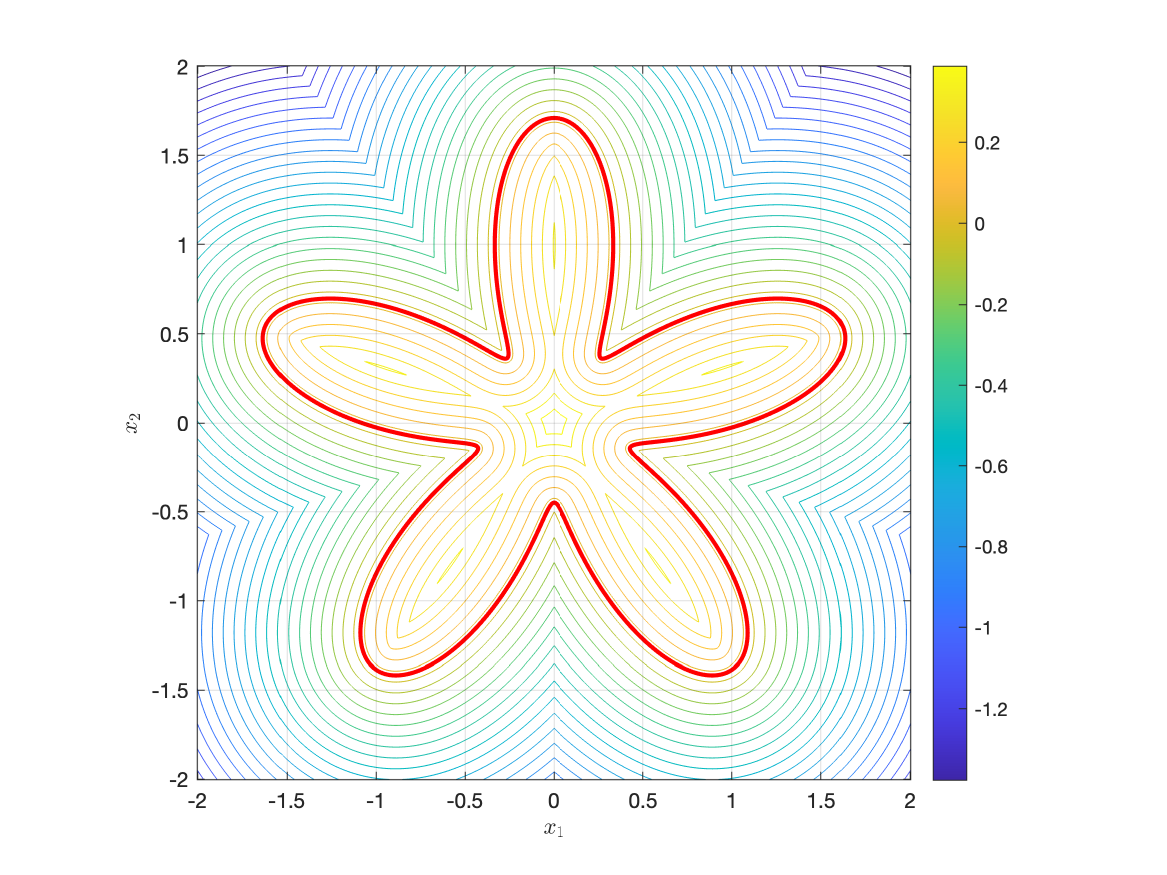}
    \caption{Plot of the boundary $\partial\Omega_1$ of the domain $\Omega_1$ defined by \qref{star-domain}, and the level curves of the signed distance function $r(x_1, x_2)$.}
    \label{fig:2DFunDomain}
\end{figure}

\begin{figure}[htb!]
\centering
\includegraphics[width=0.85\textwidth]{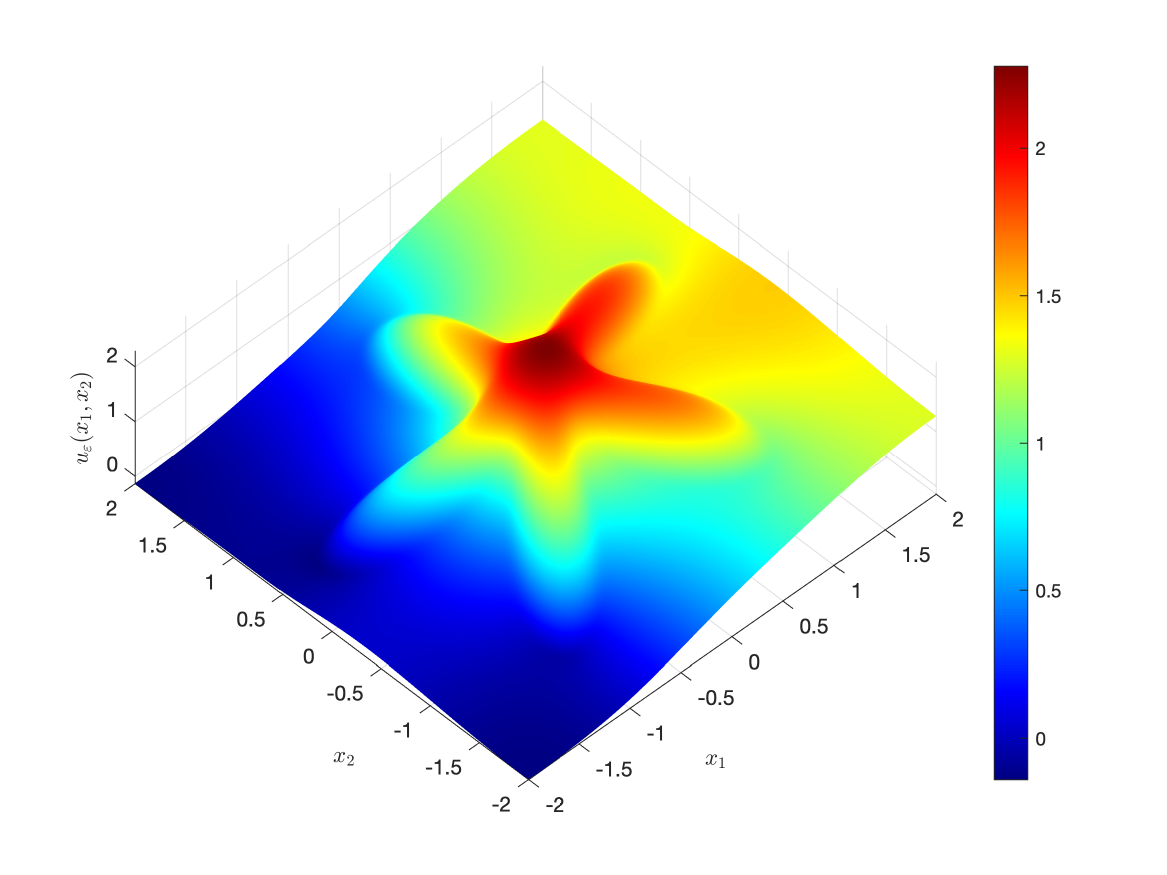}
    \caption{Plot of the solution $u_\varepsilon$ to the ``starfish'' problem, for $\varepsilon = 0.05$, with the parameters given in  \qref{starfish-1}--\qref{starfish-2}.}
    \label{fig:2DFunDomainPlot}
\end{figure}


A common concern for numerical PDE solvers is how boundary smoothness affects the rate of convergence. Corners or non-smooth features in the boundary can sometimes degrade the accuracy of traditional methods. To test whether the presence of corners impacts our diffuse domain method, we construct a test case using a convex pentagonal domain $\Omega_1$ embedded in $[-2,2]^2$. The vertices of the pentagon are chosen as
\[
(0.8,0),\quad (0.3,0.9),\quad (-0.5,0.7),\quad (-0.9,-0.2),\quad (0.2,-0.8).
\]
We define the computational domain $\Omega := [-2,2]\times[-2,2]$ and the outer domain $\Omega_2 := \Omega \setminus \overline{\Omega_1}$. The parameters used for the diffuse domain BVP \eqref{diff-bvp1}--\eqref{diff-bvp2} are:
\begin{align}
    \label{PentagonParameter-1}
\alpha &=1 , \quad \beta = 1, \quad \gamma = 2, \quad \kappa = 0.5, \quad \lambda = 4.5,
    \\
\label{PentagonParameter-2} q(x_1,x_2) & = x_1^2+5,
    \quad 
 h(x_1,x_2)  = -\sin{(x_1)}-e^{\cos{(x_2)}}.
    \end{align}

Since constructing an exact analytical solution for this domain is challenging, we define a high-accuracy numerical ground truth solution, denoted by $u_{\rm g}$. This $u_{\rm g}$ is obtained by solving the system with a very small interface width parameter, specifically $\varepsilon=0.004$. We then compute the $L^2$ and $L^{\infty}$ errors of the solutions $u_\varepsilon$ (obtained with larger $\varepsilon$ values) with respect to $u_{\rm g}$.

Figure~\ref{fig:PentagonDomainPlot} illustrates the pentagonal boundary ($\partial\Omega_1$) and the level curves of the corresponding signed distance function $r(x_1,x_2)$. The numerical ground truth solution $u_{\rm g}$ (i.e., $u_\varepsilon$ for $\varepsilon = 0.004$) is displayed in Figure~\ref{fig:PentagonDomainSolutionPlot}, showing the solution's behavior near the pentagon boundary. The convergence results, plotting the error against $\varepsilon$ on a log-log scale, are presented in Figure~\ref{fig:PentagonCauchyPlotLinf} ($L^\infty$ error) and Figure~\ref{fig:PentagonCauchyPlotL2} ($L^2$ error).
In contrast to some discretization methods where boundary singularities can significantly reduce the order of convergence, we do not observe such degradation here compared to results on smoother domains. This suggests that, for this problem and parameter regime, boundary smoothness plays a relatively minor role in the convergence behavior of the diffuse domain method with respect to $\varepsilon$. However, more testing and analysis needs to be done to determine the limitations of the method.

\begin{figure}[htb!]
    \centering
    \includegraphics[width=0.8\textwidth]{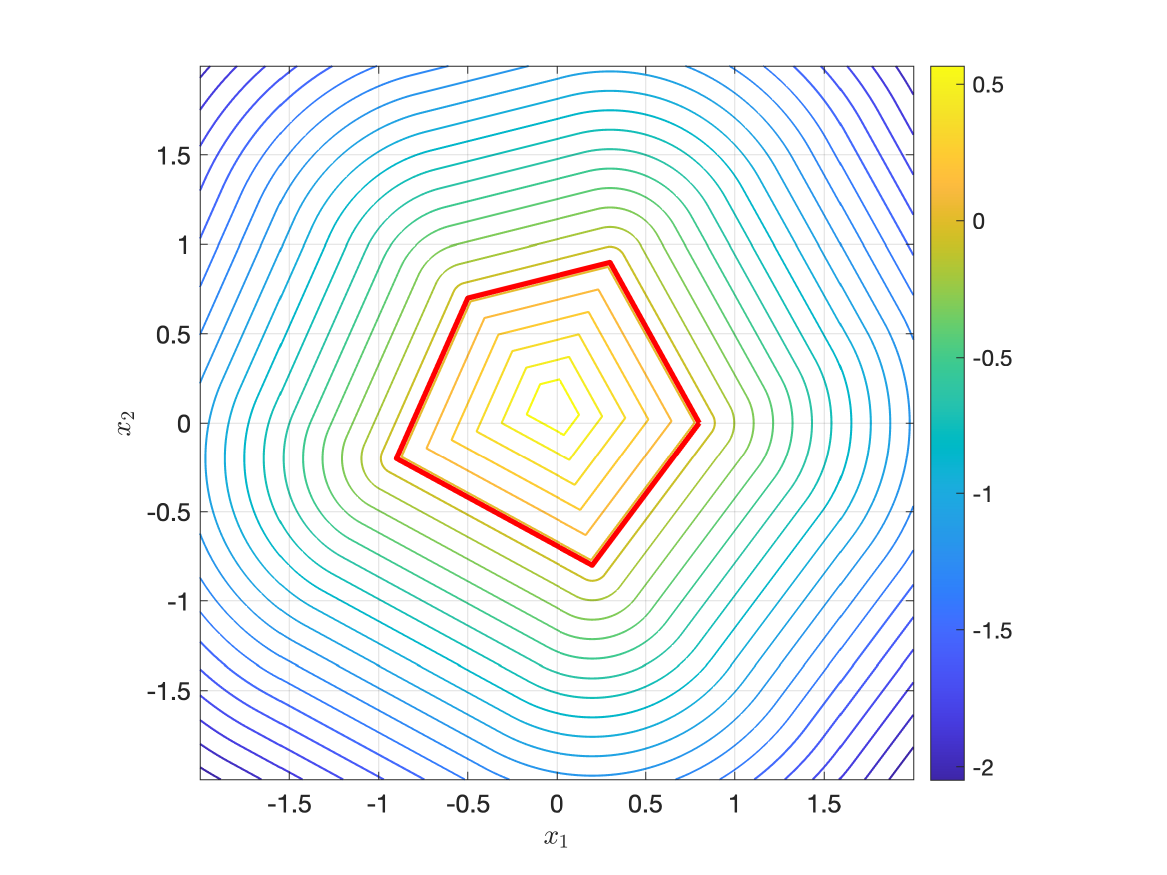}
    \caption{Plot of the pentagon domain $\Omega_1$ embedded in $[-2,2]^2$, defined by the five vertices
    $(0.8,0)$,
    $(0.3,0.9)$,
    $(-0.5,0.7)$,
    $(-0.9,-0.2)$,
    $(0.2,-0.8)$.
    The level curves show the signed distance function $r(x_1,x_2)$.} 
    \label{fig:PentagonDomainPlot}
\end{figure}

\begin{figure}[htb!]
    \centering
    \includegraphics[width=0.8\textwidth]{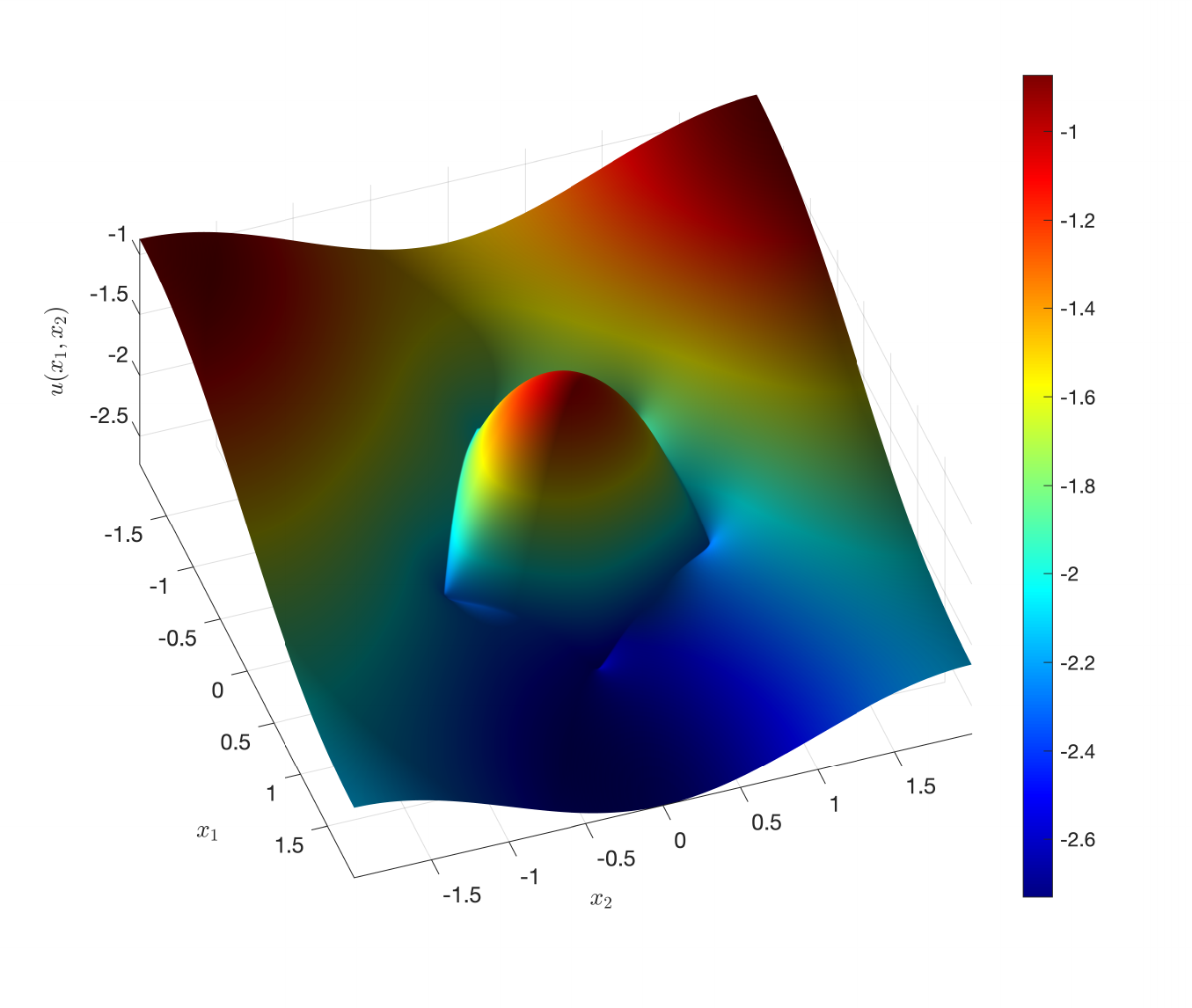}
    \caption{Plot of the diffuse domain approximation solution $u_\varepsilon$ for $\varepsilon = 0.004$ (defined as $u_{\rm g}$) in the pentagon domain.} 
    \label{fig:PentagonDomainSolutionPlot}
\end{figure}

\begin{figure}[htb!]
    \centering
    \includegraphics[width=0.75\textwidth]{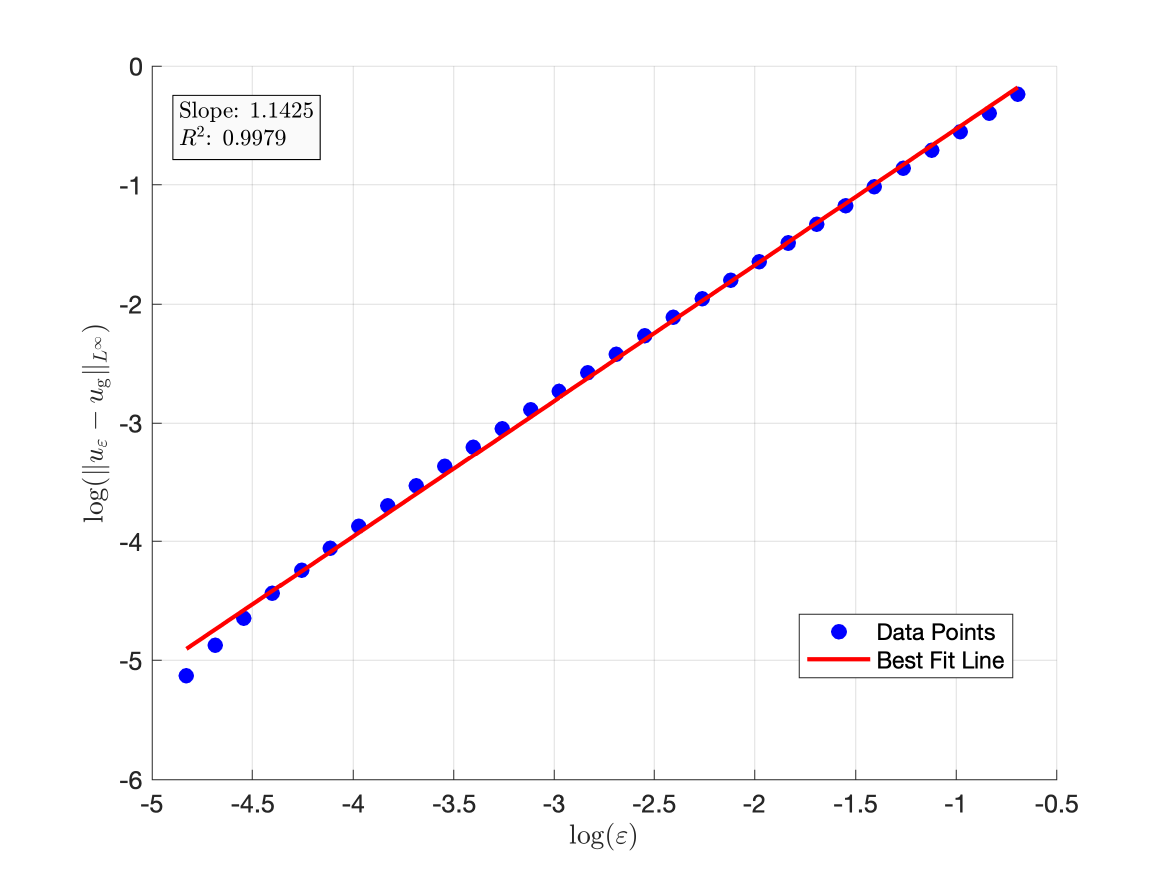}
    \caption{Log-log plot of the $L^\infty$ error $\|u_\varepsilon - u_{\rm g}\|_{L^\infty}$ versus $\varepsilon$ for the pentagon domain test. The observed order of convergence as $\varepsilon\rightarrow0$ is approximately $O\left(\varepsilon^{1.14}\right)$. No significant degradation in the convergence rate is observed compared to smoother domains.}
    \label{fig:PentagonCauchyPlotLinf} 
\end{figure}

\begin{figure}[htb!]
    \centering
    \includegraphics[width=0.75\textwidth]{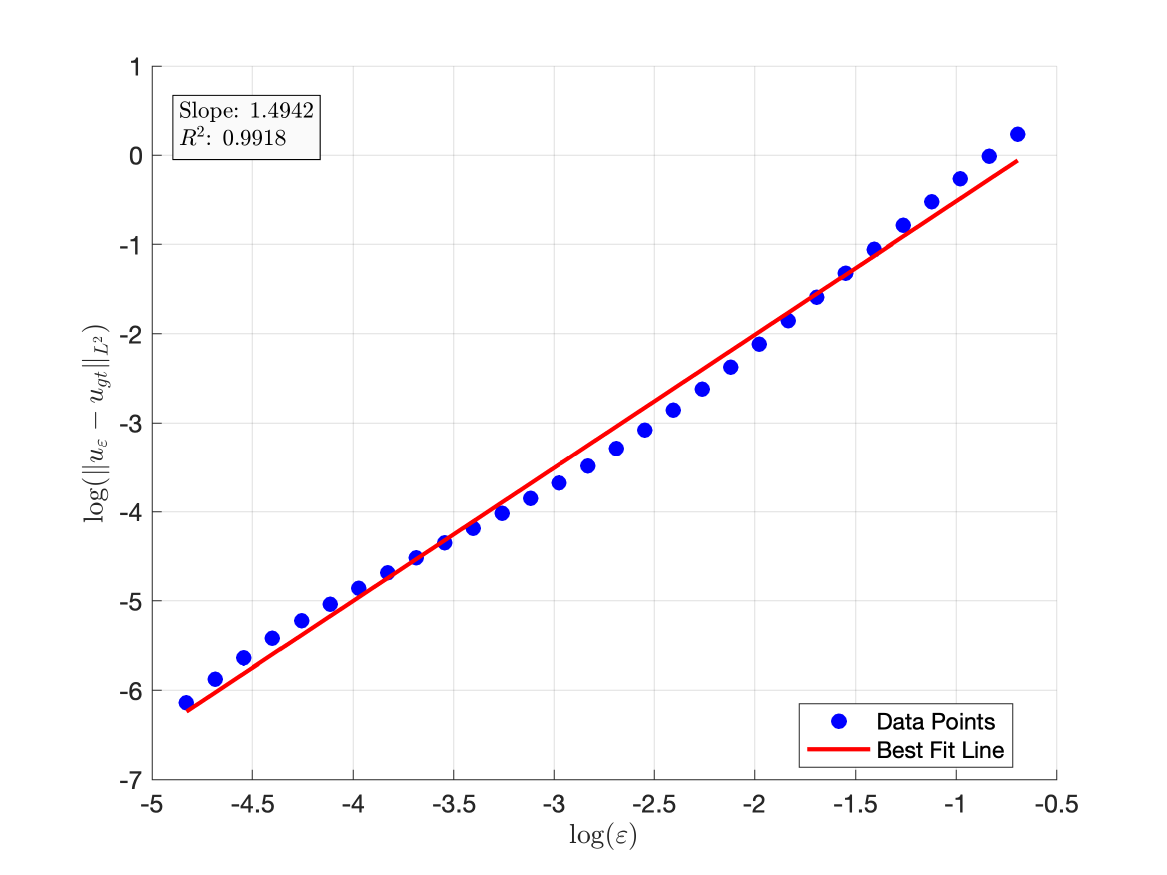}
    \caption{Log-log plot of the $L^2$ error $\|u_\varepsilon - u_{\rm g}\|_{L^2}$ versus $\varepsilon$ for the pentagon domain test. The observed order of convergence as $\varepsilon\rightarrow0$ is approximately $O\left(\varepsilon^{1.49}\right)$. No significant degradation in the convergence rate is observed compared to smoother domains.}
    \label{fig:PentagonCauchyPlotL2}
\end{figure}


    \section{Discussion}
    \label{sec:discuss}
In this work, we performed a matched asymptotic analysis of the diffuse domain approximation for the two-sided reaction-diffusion equation with general transmission boundary conditions. 
The two-sided problem is interesting and useful on its own merits. Indeed, many physical processes have different physical parameters on different sides of a dividing interface. 

But, the one-sided problem is, arguably, much more interesting, and it is the reason the diffuse domain method was developed. Clearly, to compute the solution to the one-sided problem, one must solve a well-posed problem on the whole of the cuboid $\Omega$, such that the portion of the approximation in the exterior domain, $\Omega_2$, has a negligible effect on the approximation in $\Omega_1$.  The two-sided problem we examined here was derived from the one-sided problem by introducing a modified cutoff phase-field function
    \begin{align*}
D_\varepsilon(x) = \alpha + (1 - \alpha)\phi_\varepsilon(x),
    \end{align*}
where the constant $\alpha$ is independent of $\varepsilon$ to avoid degeneracy in the exterior domain $\Omega_2$. Lerv{\aa}g and Lowengrub~\cite{Lowengrub-DDM2015} took $\alpha>0$ to be small, but still positive, and independent of $\varepsilon$. While this did not seem to spoil the second-order asymptotic convergence that they predicted, their analysis does not, strictly speaking, cover this case. Our present analysis does cover this case, as long as $\alpha > 0$ is finite and independent of $\varepsilon$. Moreover, our analysis shows that this choice degrades the asymptotic convergence to exactly first-order.

To verify this numerically, we perform the following experiments in the one-dimensional space. Specifically, we define the domains as 
    \begin{align*}
    \Omega:=(-1,1), \quad \Omega_1=\Omega_R:=(0,1), \quad \Omega_2=\Omega_L:=(-1,0), 
    \end{align*}
and choose the function
\begin{align}
    u_0(x) = \begin{cases}	 
    u_R(x) = (4x^2-8x+6)\cos{(4\pi x)} ,&  \text{ if }  x\in\Omega_R,
    \\
    u_L(x) = 6 ,&  \text{ if } x\in \Omega_L, 
    \label{1D-OneSidedTrueSolution}
    \end{cases}
\end{align}
which is a solution to the two-sided problem \qref{1D-bvp1}--\qref{1D-bvp5}, where
    \begin{align}
    \label{OneSidedParameter-1}
\alpha &>0 , \quad \beta = 0, \quad \gamma = 1, \quad \kappa = 1.6, \quad \lambda = -17.6, 
    \\
q(x) & = (16\pi^2 + 1) (4x^2 - 8x + 6)\cos(4 \pi x) + 64 \pi (x - 1) \sin(4 \pi x), 	 
    \\
\label{OneSidedParameter-2} h(x) & = 0.
    \end{align}
In this particular problem, the flux across the interface $x = 0$ exhibits a nonzero jump, satisfying
    \begin{align*}
u_R' (0) - \alpha u_L' (0) =  -8.
    \end{align*}
Picking $\alpha=0.01$,  taking the number of cell centered finite difference points to be $N=2^{15}$, and setting $\varepsilon=0.001$, our numerical results approximate solution matches $u_0(x)$ very well in the eyeball norm. See Figure~\ref{fig:OneSidedSolution}.

To quantify the error of our approximation,  we pick two small values of $\alpha$, $\alpha=0.01$ and $\alpha=0.001$, and we run the solver for several small $\varepsilon$ values. The log-log plots of the $L^2$ errors for $u_R$ against $\varepsilon$ are shown in Figures~\ref{fig:OneSided_a=0.01} and \ref{fig:OneSided_a=0.001}. Note that the sudden spikes in the errors for very small values of $\varepsilon$ are due to the finite accuracy of the numerical approximation via the finite difference method. Otherwise, we clearly observe that the convergence behavior transitions from order second-order, when $\varepsilon$ is relatively large compared to $\alpha$, to first-order, when $\varepsilon$ is relatively small compared to $\alpha$. Lerv{\aa}g and Lowengrub~\cite{Lowengrub-DDM2015} took $\alpha = 10^{-6}$, a much smaller value than the $\alpha$ values we have chosen. Consequently, they may not have observed this trend if the values of $\varepsilon$ they selected were not sufficiently small relative to $\alpha$. However, this trend will be observed as long as $\alpha > 0$ is finite and independent of $\varepsilon$. 

Now, some open questions arise: What if $\alpha$ depends on $\varepsilon$? More specifically, can the asymptotic convergence analysis be extended to the case that $\alpha = \varepsilon^m$, where $m$ is some positive number? Our preliminary computations suggest that, setting $\alpha=\varepsilon^m$, with $m\ge 2$, $\beta = 0$, and $h\equiv 0$ results in an approximation method for the one-sided problem that is numerically well-posed and well-conditioned, and one that converges to the solution of the one-sided problem at the rate of $O(\varepsilon^2)$. 

In Figures~\ref{fig:OneSided_a=e^2} and \ref{fig:OneSided_a=e^3}, we plot the errors for the problem described above, where the only difference is that we take $\alpha=\varepsilon^2$ and $\alpha=\varepsilon^3$, respectively. Full second-order convergences are observed in both cases. What happens to the solution in the exterior domain, $\Omega_2$, in this case? Our preliminary tests suggest that the solution converges to a non-trivial function, one that is straightforward  to characterize. In one space dimension, and using the setup described above,  the solution $u_2$ over $\Omega_2$ will simply be a constant function, similar to what is shown in Figure~\ref{fig:OneSidedSolution}. However, rigorously proving that this limit is unique and well-defined is much more challenging.

When we make such an alteration --- that is, when we take $\alpha=\varepsilon^m$, $m\ge 2$, $\beta = 0$, and $h\equiv 0$ --- while we ensure that $D_\varepsilon$ does not decay exponentially rapidly in the exterior domain $\Omega_2$, we introduce significant, new mathematical difficulties for the asymptotic analysis. 
We will report on this challenging and intriguing one-sided problem in a forthcoming paper.

    \begin{figure}[htb!]
    \centering
\includegraphics[width=0.7\linewidth]{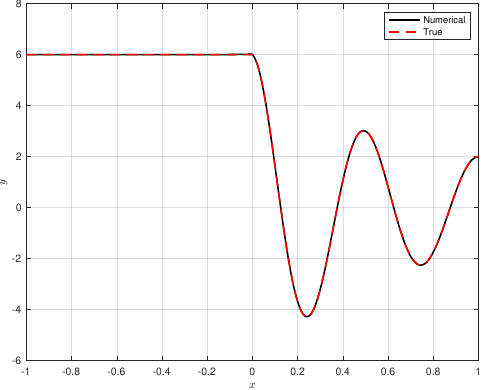}
\caption{Plots of the true solution $u_0$ specified by \qref{1D-OneSidedTrueSolution} and the diffuse domain approximation solution $u_\varepsilon$ to the 1D problem, for $\varepsilon=0.05$, with the parameters given in  \qref{OneSidedParameter-1}--\qref{OneSidedParameter-2}.}
    \label{fig:OneSidedSolution}
    \end{figure}

\begin{figure}[htb!]
    \centering
    \begin{subfigure}{0.49\textwidth}
        \centering
        \includegraphics[width=0.95\linewidth]{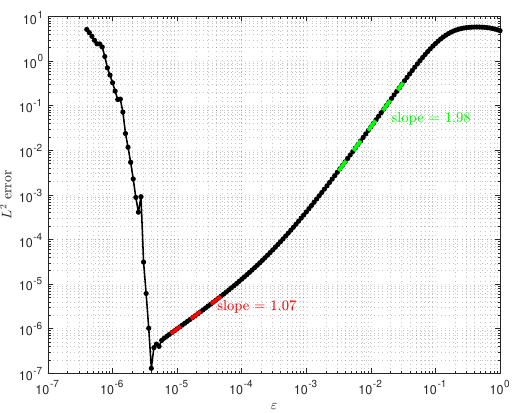}
        \caption{$\alpha=0.01$.}
        \label{fig:OneSided_a=0.01}
    \end{subfigure}
    \begin{subfigure}{0.49\textwidth}
        \centering
        \includegraphics[width=0.95\linewidth]{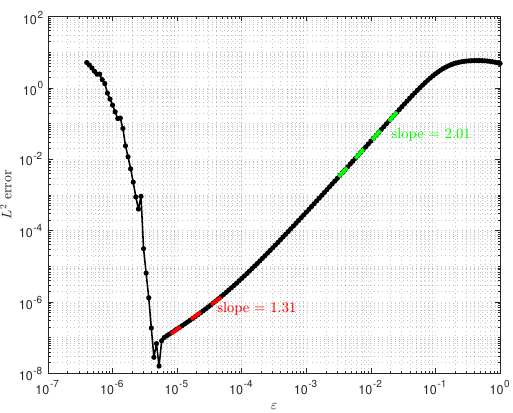}
        \caption{$\alpha=0.001$.}
        \label{fig:OneSided_a=0.001}
    \end{subfigure}
    \caption{The log-log plots of the $L^2$ Error for $u_R$ against $\varepsilon$ when $\alpha$ is independent of $\varepsilon$.}
    \label{fig:OneSided_a=const}
\end{figure}

\begin{figure}[htb!]
    \centering
    \begin{subfigure}{0.49\textwidth}
        \includegraphics[width=0.95\linewidth]{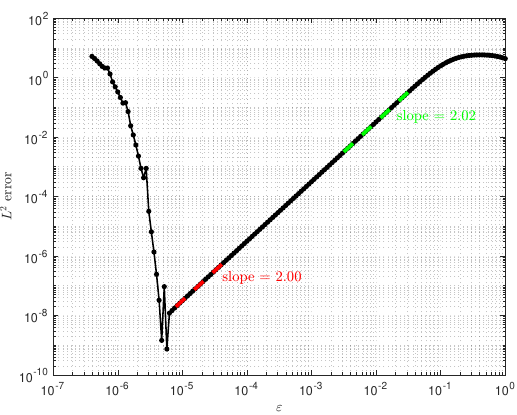}
        \caption{$\alpha=\varepsilon^2$.}
        \label{fig:OneSided_a=e^2}
    \end{subfigure}
    \begin{subfigure}{0.49\textwidth}
        \centering
        \includegraphics[width=0.95\linewidth]{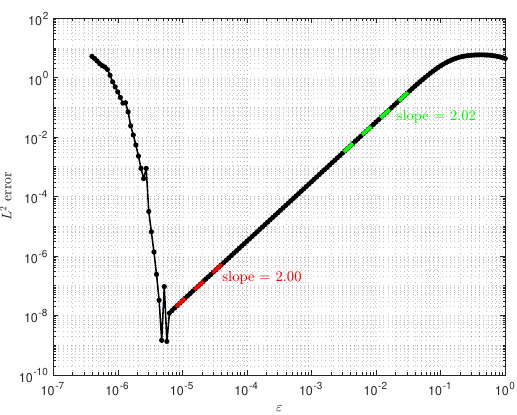}
        \caption{$\alpha=\varepsilon^3$.}
        \label{fig:OneSided_a=e^3}
        \end{subfigure}
    \caption{The log-log plots of the $L^2$ Error for $u_R$ against $\varepsilon$ when $\alpha$ depends on $\varepsilon$.}
    \label{fig:OneSided_a=e^m}
\end{figure}

    \section*{Acknowledgments}

TM gratefully acknowledges funding from the US National Science Foundation via grant number NSF-DMS 2206252. TL is thankful for partial support from the Department of Mathematics at the University of Tennessee. SMW and MHW acknowledge generous support from the US National Science Foundation through grants NSF-DMS 2012634 and NSF-DMS 2309547. The authors thank John Lowengrub for helpful discussions about this topic.

    \bibliographystyle{plain}
    \bibliography{Luong-bib}

    \end{document}